%% file: main.tex
\documentclass[11pt]{amsart}

\usepackage[utf8]{inputenc}
\usepackage{amsmath,amssymb,enumerate,cite}
\usepackage{wrapfig}
\usepackage[font=small]{caption, subcaption}
\usepackage{xcolor}
\usepackage[normalem]{ulem}
\usepackage{dirtytalk}
\usepackage{graphicx}
\usepackage{pdflscape}
\usepackage{rotating}

\newtheorem{theorem}{Theorem}[section]
\newtheorem{lemma}{Lemma}[section]

\theoremstyle{definition}
\newtheorem{remark}{Remark}[section]
\newtheorem{definition}{Definition}[section]
\theoremstyle{definition}
\newtheorem{example}{Example}[section]

\title{Combinatorial Geometry of Threshold-Linear Networks}
\date{August 3, 2020}

\begin{document}
\author{Carina Curto}
\author{Christopher Langdon} 
\author{Katherine Morrison}
\maketitle

\input{Section-0.tex}
\input{Section-1.tex}
\input{Section-2.tex}

\input{Section-3.tex}

\input{Section-4.tex}

\input{Section-5.tex}

\bibliographystyle{abbrv}
\bibliography{ref}
\end{document}

%% file: Section-0.tex
\noindent {\bf Abstract.} The architecture of a neural network constrains the potential dynamics that can emerge.  Some architectures may only allow for a single dynamic regime, while others display a great deal of flexibility with qualitatively different dynamics that can be reached by modulating connection strengths.  In this work, we develop novel mathematical techniques to study the dynamic constraints imposed by different network architectures in the context of competitive threshold-linear networks (TLNs).  Any given TLN is naturally characterized by a hyperplane arrangement in $\mathbb{R}^n$, and the combinatorial properties of this arrangement determine the pattern of fixed points of the dynamics. This observation enables us to recast the question of network flexibility in the language of oriented matroids, allowing us to employ tools and results from this theory in order to characterize the different dynamic regimes a given architecture can support.  In particular, fixed points of a TLN correspond to cocircuits of an associated oriented matroid; and mutations of the matroid correspond to bifurcations in the collection of fixed points.  As an application, we provide a complete characterization of all possible sets of fixed points that can arise in networks through size $n=3$, together with descriptions of how to modulate synaptic strengths of the network in order to access the different dynamic regimes. These results provide a framework for studying the possible computational roles of various motifs observed in real neural networks.

%% file: Section-1.tex
\vspace{.05in}

\section{Introduction}
 An $n$-dimensional threshold-linear network (TLN) is a continuous dynamical system given by the nonlinear equations:
\begin{equation}
\dot{x}_i=-x_i+\bigg[\sum_{j= i}^n W_{ij}x_j+b_i\bigg]_+ \hspace{.5cm}i=1,\dots,n
\end{equation}
where $\dot{x}_i = dx_i/dt$ denotes the time derivative, $b_i$ and $W_{ij}$ are all real numbers, $W_{ii} = 0$ for each $i$, and $[x]_+=\text{max}\{x,0\}$. TLNs serve as simple models of recurrent neural circuits. The $i$-th component is interpreted as a simplified description of the dynamics of a single neuron $i$ where $x_i(t)$ represents the firing rate as a function of time, $b_i$ the external input to the $i$-th neuron and $W_{ij}$ the strength of the synaptic connection between the presynaptic $j$ and the postsynaptic $i$. In this work we will consider competitive networks with positive external input so that $W_{ij}<0$ and $b_i>0$ for all $i$ and $j$ when $i \neq j$ (and $W_{ii} = 0$). These networks were previously studied by the authors in \cite{fp-paper, robust-motifs, CTLN-preprint}.

\par For a fixed point $x\in\mathbb{R}^n$, the subset $\sigma=\{i\in \{1,\ldots,n\} :x_i>0\}$ is called the \textit{support} of the fixed point and corresponds to the set of coactive neurons at the fixed point. Using the notation $[n] = \{1,\ldots,n\}$ we see that $\sigma \subset [n]$. The supports corresponding to stable fixed points are called permitted sets and have been previously studied in \cite{HahnSeungSlotine}. Here we consider both stable and unstable fixed points, as unstable fixed points appear to play a significant role in the transient dynamics of these networks as well as in shaping the dynamic attractors (such as limit cycles) \cite{fp-paper, robust-motifs}. For a network defined by $W$ and $b$, the set of all fixed point supports is denoted $\operatorname{FP}(W,b)$. \\

\noindent{\it A piecewise-linear dynamical system}
\par An alternative perspective on fixed point supports is obtained by considering a threshold-linear network as a continuous piecewise-linear dynamical system. Indeed, if for each $i=1,...,n$ we define on the state space the linear functional $l_i^*:\mathbb{R}^n\to\mathbb{R}$:
\begin{equation}
l^*_i(x):=\sum_{j\not=i}W_{ij}x_j+b_i \hspace{.5cm}i=1,\dots,n
\end{equation}
then the zero sets $L_i:=\{l^*_i(x)=0\}$ are hyperplanes that partition the state space into cells where the equations of the dynamical system are linear.
Specifically, each $L_i$ partitions $\mathbb{R}^n$ into a positive and negative side according to the sign of $l^*_i$ and for each subset $\sigma\subset[n]$ we have the cell $L^\sigma$ defined by $l^*_i>0$ for $i\in\sigma$ and $l^*_i<0$ for $i\not\in\sigma$. The restriction of $(1)$ to this cell is linear and given by the equations
\vspace{.05in}
\begin{equation}\label{eq:1}
\dot{x}_i=\begin{cases}
-x_i +\displaystyle{\sum_{j\not=i}W_{ij}x_j+b_i} \hspace{.8cm} i\in\sigma\\
-x_i\hspace{3.4cm} i\not\in\sigma
\end{cases}
\end{equation}
\vspace{.05in}

 Generically, these linear systems have a single fixed point which we denote $x^\sigma\in\mathbb{R}^n$.  If $x^\sigma\in L^\sigma$ then $x^\sigma$ is said to be \textit{admissible} and is a fixed point of the nonlinear system $(1)$. If $x^\sigma\in L^\tau$ where $\tau\not=\sigma$, then it is said to be \textit{virtual} and it is not a fixed point of the nonlinear system. If $x^\sigma$ is admissible then the support of $x^\sigma$ is $\sigma $ and thus the set of supports $\operatorname{FP}(W,b)$ can equivalently be thought of as the admissible subsets.\\

\noindent{\it Support bifurcations}
  \par A key point is that the set $\operatorname{FP}(W,b)$ can change upon variation of the network parameters $W$ and $b$. We will call such a scenario a \textit{support bifurcation}. Alternatively, support bifurcations can be viewed as boundary bifurcations, bifurcations that arise in piecewise-linear dynamical systems and correspond to admissible equilibria crossing the boundaries between linear systems. 

 \begin{figure}[h]
    \centering
    \includegraphics[width=5in]{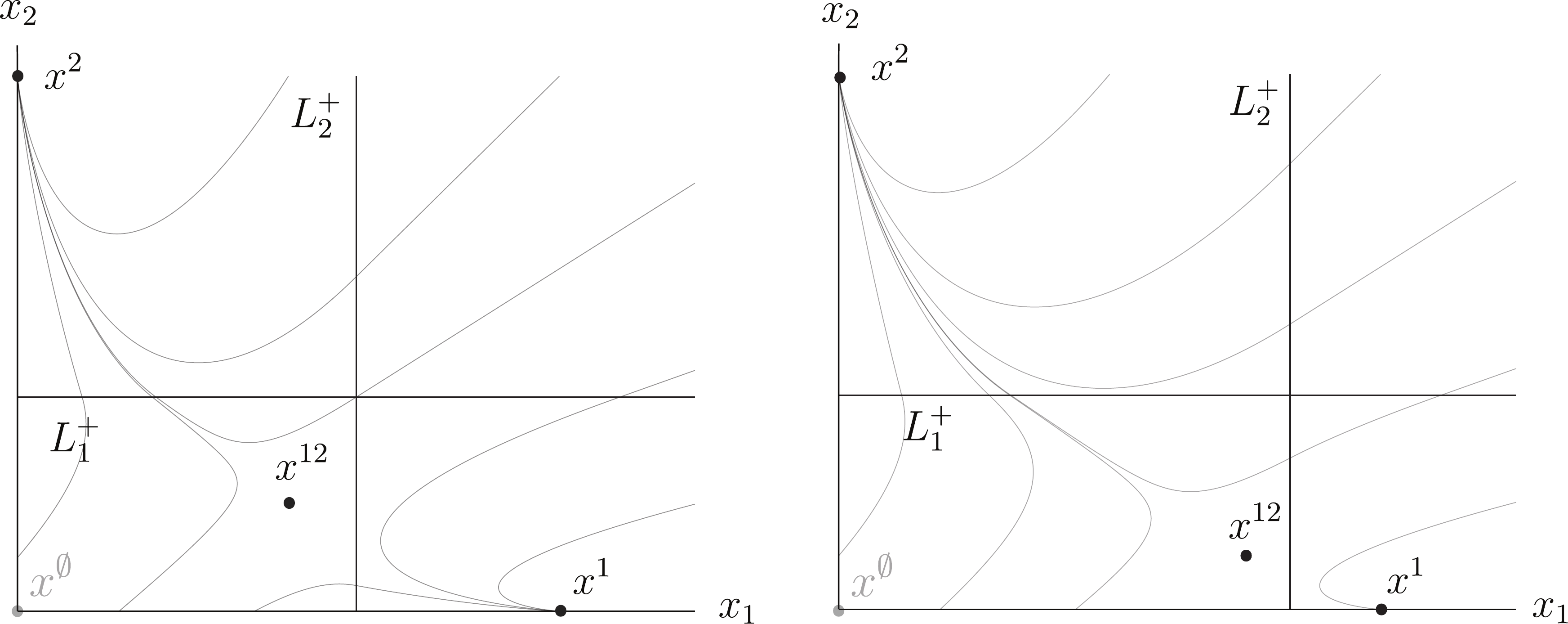}
    \caption{The fixed point supports of this network are $\{1,2,12\}$. The support bifurcation $\{1,2,12\}\to\{2\}$ occurs when the admissible fixed points $x^{12}$ and $x^1$ collide at the boundary $L_2$ and become virtual. }
    \label{fig:switching}
\end{figure}
\vspace{-.05in}

\par  Boundary bifurcations in general piecewise-smooth systems have been studied  in \cite{benardo} and it is known, for instance, that the fixed point either persists or disappears in a non-smooth fold scenario (Figure~\ref{fig:switching}). In this work we study constraints on the sets of admissible fixed point supports and how modulation of the recurrent connectivity and the external input leads to support bifurcations. \\

\noindent {\it The graph of a TLN}
\par The relationship between recurrent connectivity and fixed points has been studied in \cite{CTLN-preprint}, \cite{robust-motifs} and \cite{fp-paper} by considering the quantities $s^{ij}_j:=b_iW_{ji}+b_j$ and encoding their signs in a directed graph.
\begin{definition} Consider a TLN specified by $(W,b)$. For an ordered pair of neurons $(i,j)$ we define the quantity:
\[
s^{ij}_j:= b_iW_{ji}+b_j
\]
and define a directed graph by the rule $i\to j$ if and only if $s^{ij}_j>0$. This is the graph of the TLN.
\end{definition}
Note that in the case where $b_i = \theta > 0$ for all $i$, the graph has an edge $i \to j$ if and only if $W_{ji}>-1$.
The graph of a TLN captures pairwise asymmetries in the network connectivity and reveals strong constraints on the admissible sets of fixed point supports. Here we will show that the quantities $s^{ij}_j$ arise from a particular hyperplane arrangement associated to the network and obtain additional geometric constraints on admissible sets of supports. The overall goal of this work is to determine how the graph constrains the possibilities for the fixed points supports $\operatorname{FP}(W,b)$.\\

\noindent {\it Classification of TLN dynamic regimes for $n=3$}
\par In dimension $n=3$, this leads naturally to the consideration of the following quantities.
 \begin{definition}
 For each ordered triple of neurons $(i,j,k)$ we define the quantity:
 \[
 \Delta^{ij}_k:=b_jW_{ik}-b_iW_{jk}
 \]
 \end{definition}
The signs of these quantities can be viewed as an extension of the directed graph capturing higher order asymmetries in the network connectivity. We show that the bifurcation picture in dimension three has a simple description in terms of modulation of the quantities $s^{ij}_j$ and $\Delta^{ij}_k$.

If we fix the external inputs $b_i$, then to vary a quantity $\Delta^{ij}_k$ it is sufficient to vary the difference of the connectivity weights $W_{ik} - W_{jk}$. In this case, the relevant quantities for graphs on three nodes are simply:
 \[
\Delta W^{ij}_k := W_{ik} - W_{jk}.
\]
These are bifurcation parameters that enable us to transition between different dynamic regimes in $n=3$. Figure~\ref{fig:n3-classification} provides a summary of our results for the 16 directed graphs on three vertices. For each graph, dynamic regimes are given for the set of all competitive TLNs $(W,b)$ having the specified graph, as prescribed above. Gray nodes denote different dynamic regimes and are labeled by the set of fixed point supports
$\operatorname{FP}(W,b)$ that appear for all TLNs in that regime. By changing the quantities $\Delta W^{ij}_k$, as denoted by the blue arrows, one can transition from one dynamic regime to another while preserving the overall architecture of the graph. In some cases, there is only a single dynamic regime. This means all TLNs for the given graph have the same set of fixed point supports. Such graphs are called {\it robust motifs} \cite{robust-motifs}, and are given in the red box. Finally, tan nodes denote the sets of fixed point supports for the combinatorial threshold-linear network (CTLN) corresponding to each graph. CTLNs are a special sub-family of TLNs that were introduced in \cite{CTLN-preprint, fp-paper}, but they are not a focus of this paper.\\

The organization of this paper is as follows. In Section 2 we describe the hyperplane arrangements associated to TLNs, and introduce oriented matroid language that captures combinatorial properties of these arrangements in a geometric picture. 
In Section 3 we introduce the chirotope of a TLN and use it to characterize the admissible fixed point supports.  In Section 4, we relate support bifurcations to mutations of chirotopes.  Finally, in Section 5 we apply the oriented matroid tools to fully classify the possible sets of fixed point supports that can arise for each directed graph of size $n=3$. This is where we prove the various results that are summarized in Figure~\ref{fig:n3-classification}. 

\begin{landscape}
 \begin{figure}[!h]
 \vspace{.5in}
 \begin{center}
    \includegraphics[width=8.75in]{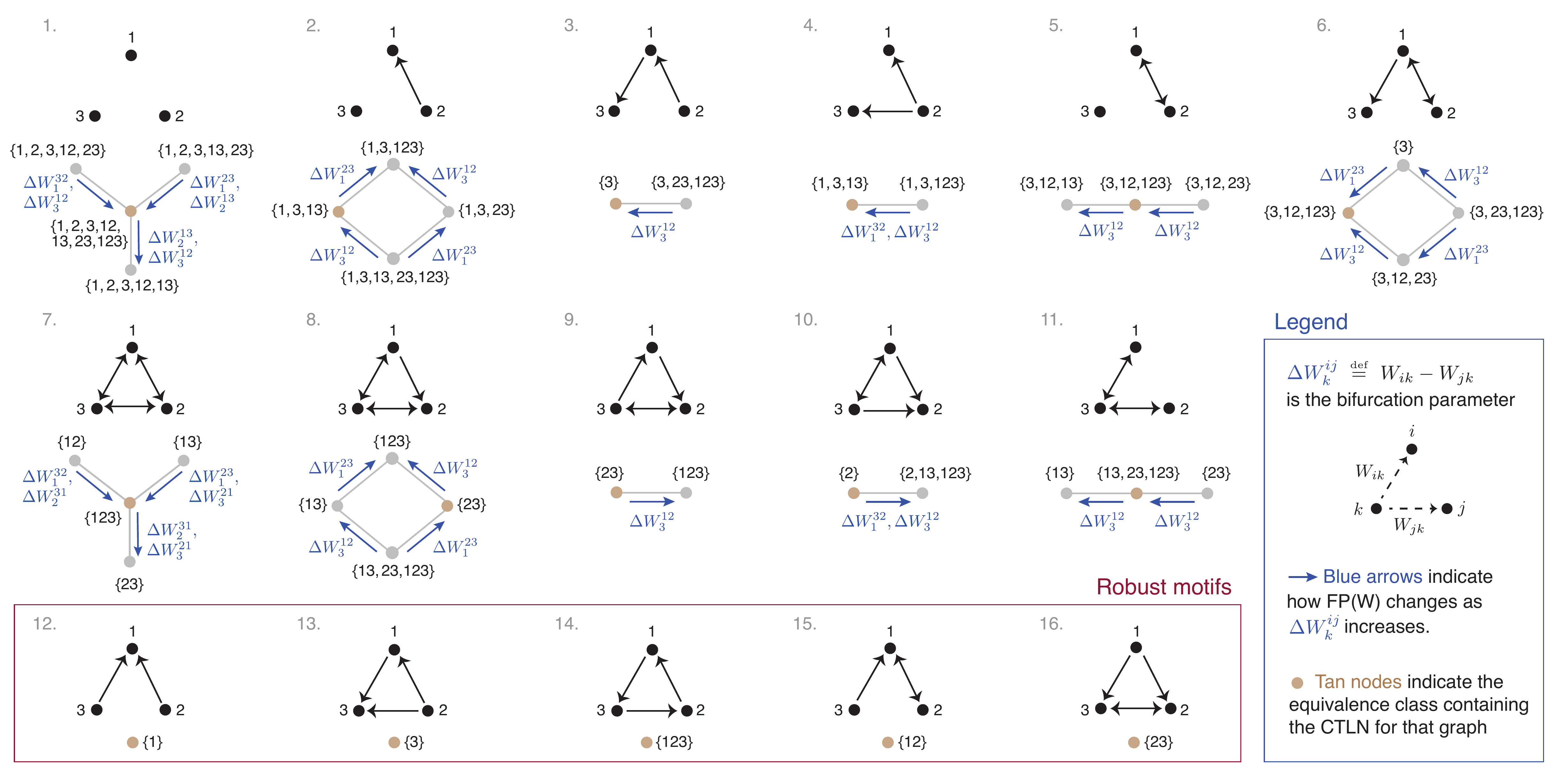}
    \end{center}
    \caption{Classification of dynamic regimes for competitive TLNs on $n=3$ nodes. For each network $(W,b)$, the 
    possible sets of fixed point supports depend on the graph of the TLN. For each graph, the sets of possible supports are given together with prescriptions for how to change the synaptic weights in $W$ to move from one dynamic regime to another, while preserving the network architecture. All graphs provide strong constraints on the possible network dynamics, allowing at most four dynamic regimes. Robust motifs (graphs 12-16) have only a single dynamic regime.}
    \label{fig:n3-classification}
\end{figure}
\end{landscape}

%% file: Section-2.tex
\section{Cocircuits of a threshold-linear network}
In this section we identify an affine hyperplane arrangement associated with any threshold-linear network and show that the combinatorial equivalence class of this arrangement is sufficient to determine the fixed point supports of the network. Specifically, we show that the set of cocircuits of the arrangement determines $\operatorname{FP}(W,b)$.  It follows that support bifurcations correspond to changes in the cocircuits of the arrangement. We establish a correspondence between network parameters and the geometry of the arrangement by connecting the quantities $b_i$, $b_iW_{ji}+b_j$ and $b_iW_{kj}-b_kW_{ij}$ with its cocircuits. Such a connection will allow an understanding of how modulation of network parameters leads to qualitative changes in the network dynamics.

\par To obtain an affine arrangement from a threshold-linear network, consider the linear functionals:
\begin{equation}
\begin{split}
h^*_i(x)&:=-x_i+\sum_{j\not=i}W_{ij}x_j+b_i\\
e^*_i(x)&:=x_i
\end{split}
\end{equation}
for $i=1,...,n$. By taking the zero sets $E_i:=\{e^*_i(x)=0\}$ and $H_i:=\{h^*_i(x)=0\}$ we obtain an arrangement of $2n$ affine hyperplanes in the state space $\mathbb{R}^n$ which we will denote by $\mathcal{A}=\mathcal{A}(W,b)$. Note that the relationship with the linear boundaries is given by the equation:
\begin{equation}\label{glueing}
h^*_i(x)=-e^*_i(x)+l^*_i(x)
\end{equation}
for $i=1,...,n$. The equations of each linear region are constructed from a subset of $\{e^*_1,h^*_1,...,e^*_n,h^*_n\}$ and in particular, the fixed point $x^\sigma$ of the linear region $L^\sigma$ is a vertex of the arrangement $\mathcal{A}$. The arrangement $\mathcal{A}$ can be viewed as the smallest hyperplane arrangement that contains the nullclines of the threshold-linear network. To encode the geometry of this arrangement we consider its cocircuits.
\begin{definition} 
The \textit{cocircuits of a threshold-linear network} we define to be the image of the map
\begin{align*}
\pi:\mathbb{R}^n&\to \{+,0,-\}^{2n}\\
x&\mapsto \big( \text{sgn}\:e^*_i(x),\text{sgn}\:h^*_i(x) \big)_{i=1,...,n}
\end{align*}
restricted to the vertices of the arrangement $\mathcal{A}$.
\end{definition}
For each vertex $x$ in the arrangement $\mathcal{A}$, the cocircuit of $x$ is a sign vector recording the position of $x$ with respect to each hyperplane in $\mathcal{A}$. In general, the collection of cocircuits of a hyperplane arrangement defines an oriented matroid (see, for instance, \cite{OMbook}) and the above definition is just a restriction to the case of arrangements arising from threshold-linear networks. Oriented matroids capture the combinatorial properties of an arrangement and two threshold-linear networks with the same set of cocircuits we will think of as combinatorially equivalent. The following lemma says that two distinct but combinatorially equivalent threshold-linear networks must have the same set of fixed point supports.
\begin{lemma}\label{the_lemma}
A subset of neurons $\sigma$ supports a fixed point if and only if the cocircuit $C^\sigma\in\{+,0,-\}^{2n}$ of $x^\sigma$ satisfies:
\[
C^\sigma(E_i)=+ \text{ and } C^\sigma(H_j)=-
\]
for $i\in\sigma$ and $j\not\in\sigma$.
\end{lemma}
\begin{proof}
By definition of $x^\sigma$, we have $h^*_i(x^\sigma)=0$ and $e^*_j(x^\sigma)=0$ for $i\in\sigma$ and $j\not\in\sigma$. From equation \eqref{glueing} with $x=x^\sigma$, we have the equalities:
\[
l^*_i(x^\sigma)=e^*_i(x^\sigma) \text{ and } l^*_j(x^\sigma)=h^*_j(x^\sigma)
\]
for $i\in\sigma $ and $j\not\in\sigma$. By taking signs, we have that $x^\sigma$ lies in the linear cell $L^\sigma$ if and only if $\text{sgn }e^*_i(x^\sigma)=+$ and $\text{sgn }h^*_j(x^\sigma)=-$.
\end{proof}
If follows that support bifurcations correspond to changes in the cocircuits of the arrangement.  Note that such a change occurs when the network parameters are varied in such a way that a hyperplane is pushed over a vertex.  To illustrate this, we revisit the bifurcation in Figure~\ref{fig:switching} but from the perspective of the arrangement $\mathcal{A}$ and its cocircuits.
\begin{example}
\begin{figure}[h]
    \centering
    \includegraphics[width=5.5in]{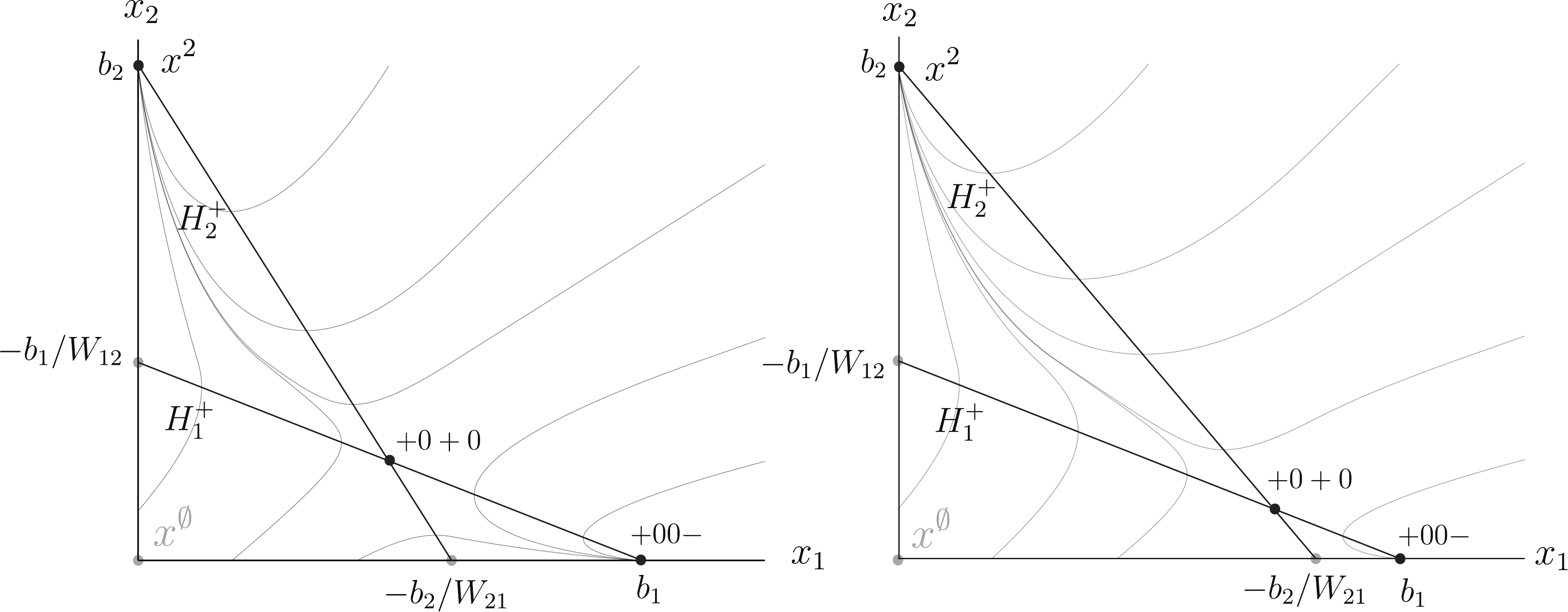}
    \caption{The support  bifurcation from Figure~\ref{fig:switching} from the perspective of the hyperplane arrangement $\mathcal{A}$. The fixed points $x^{12}$ and $x^1$ colliding at the boundary $L_2$ and crossing can equivalently be viewed as the hyperplane $H_2$ being pushed over the vertex $x^1$ changing the cocircuits of the arrangement and hence its combinatorial equivalence class. }
    \label{fig:2d_bif}
\end{figure} 
In the two dimensional case the arrangement $\mathcal{A}=\{E_1,H_1,E_2,H_2\}$ is an arrangement of four lines in the plane. From the definition of the linear functions $h^*_1$ and $h^*_2$, we see that the lines $\{H_1,H_2\}$ are oriented with the origin on their positive side (Figure~\ref{fig:2d_bif}).  This information is sufficient to determine the cocircuits of the arrangement and we see that the bifurcation in Figure~\ref{fig:switching} corresponds to the change in cocircuits:
\begin{align*}
    C^1:&\: (+,0,0,-)\to (+,0,0,+) \\
    C^{12}:&\: (+,0,+,0) \to (+,0,-,0)
\end{align*}
where each cocircuit records the position of the fixed point with respect to the ordered set of hyperplanes $\{E_1,H_1,E_2,H_2\}$. In other words, $x^{12}$ and $x^1$ crossing the boundary $L_2$ can equivalently be viewed as $x^{12}$ crossing the hyperplane $E_2$ and $x^{1}$ crossing the hyperplane $H_2$. If we observe  that the intersection of $H_2$ with $E_2$ is given by $x_1=-b_2/W_{21}$ and $x_2=0$, we see that this bifurcation can be realized by increasing the negative synaptic weight $W_{21}$ toward zero, for instance. 
\end{example}
In this way, we view support bifurcations as changes in the combinatorial geometry of the arrangement arising from modulation of network parameters. In the following three lemmas we translate the network parameters into geometric properties of the hyperplane arrangement. In particular, we work out the relationship between network parameters and the restriction of the arrangement to each coordinate axis. Although the intersection of an arrangement with the coordinate axes does not uniquely determine the arrangement, in dimension three an understanding of this geometry is sufficient for a complete understanding of support bifurcations.
\begin{lemma}\label{lemma1}
The intersection of the hyperplane $H_i$ with the $x_i$-axis lies on the positive side of $E_i$ if and only if $b_i>0$.
\end{lemma}
\begin{proof}
The intersection of $H_i$ with the $x_i$-axis is given by:
\[
x^i:=H_i\cap \bigcap_{j\not=i}E_j
\]
Since $H_i$ is given by the zero set:
\[
H_i:=\big\{-x_i+\sum_{j\not=i}W_{ji}x_j + b_i=0\big\}
\]
it follows that $x^i$ is defined by the equations $x_i=b_i$ and $x_j=0$ for $j\not=i$. Thus, $x^i$ lies on the positive side of the coordinate hyperplane $E_i$ if and only if $e^*_i(x^i)=b_i>0$. 
\end{proof}

\begin{lemma}\label{lemma2}
The intersection of $H_i$ with the $x_i$-axis lies on the positive side of $H_j$ if and only if $s^{ij}_j=b_iW_{ji}+b_j>0$.
\end{lemma}

\begin{proof}
The position of $x^i$ with respect to the hyperplane $H_j$
is obtained by evaluating $h^*_j$ at $x_i$:
\[
h^*_j(x^i) =\big( -x_j+\sum_{k\not= j}W_{jk}x_k+b_j\big )\bigg\rvert_{x^i} = b_iW_{ji}+b_j
\]
Thus, $x^i$ lies on the positive side of $H_j$ if and only if $b_iW_{ji}+b_j>0$.
\end{proof}

The above lemmas imply that a single neuron $i$ supports a fixed point only if the external input to that neuron is positive and $s^{ij}_j=b_iW_{ji}+b_j<0$ for all $j$. If we encode the signs of $s^{ij}_j=b_iW_{ji}+b_j$ in a directed graph, then we have that a single neuron supports a fixed point if and only if the external input to $i$ is positive and the node corresponding to $i$ is a sink. The bifurcation in Figure~\ref{fig:2d_bif} can be viewed from this perspective as varying the network parameters such that the edge $1\to 2$ is added to $G$.  In dimension two, all bifurcations of competitive threshold-linear networks with uniform positive external input are obtained as changes to the underlying graph $G$,\cite{robust-motifs}. In dimension three this is no longer the case and it becomes necessary to consider higher order interactions between triples of neurons.
\begin{lemma}
The intersection of $H_i$ with the $x_j$-axis lies on the positive side of $H_k$ if and only if $\Delta^{ki}_j=b_iW_{kj}-b_kW_{ij}>0$.
\end{lemma}
\begin{proof}
Let $p_{ij}$ denote the intersection of $H_i$ with the $x_j$ axis. We have
\[
p_{ij}:=H_i\cap \bigcap_{l\not=j}E_l
\]
and so $p_{ij}$ is defined by $x_j=-b_i/W_{ij}$ and $x_l=0$ for $l\not=j$. The position of $p_{ij}$ with respect to $H_k$ is then obtained by evaluating:
\[
h^*_k(p_{ij})=-x_k+\sum_{l\not=k}W_{kl}x_l+b_k\bigg\rvert_{p_{ij}}=W_{kj}(-b_i/W_{ij})+b_k
\]
Using our assumption of competitiveness and multiplying by $W_{ij}<0$, we have that $h^*_k(p_{ij})>0$ if and only if $b_iW_{kj}-b_kW_{ij}>0$.
\end{proof}

In the following example we illustrate the correspondences between network parameters and geometry established in the above lemmas by exhibiting some support bifurcations in dimension three.
\begin{example}\label{Example1} Consider the threshold-linear network defined by the parameters
\[
W=\small\begin{bmatrix}
0&  -0.97 &  -1.47\\
   -0.65  & 0 &  -0.57\\
   -1.34 &  -1.45  & 0\\
\end{bmatrix}
\hspace{.25cm} b=\begin{bmatrix}
0.49\\
    0.40\\
    0.62
\end{bmatrix}
\]
In dimension three the arrangement $\mathcal{A}$ consists of six hyperplanes in $\mathbb{R}^3$. The arrangement for this particular network is shown   Figure~\ref{Arrangements 1}.  The hyperplanes $\{E_1,H_1\}$, $\{E_2,H_2\}$ and $\{E_3,H_3\}$ are colored blue, red and yellow, respectively, and $H_1$, $H_2$ and $H_3$ are all oriented with the origin on the positive side. 
  \begin{figure}[h]
    \centering
    \includegraphics[width=6in]{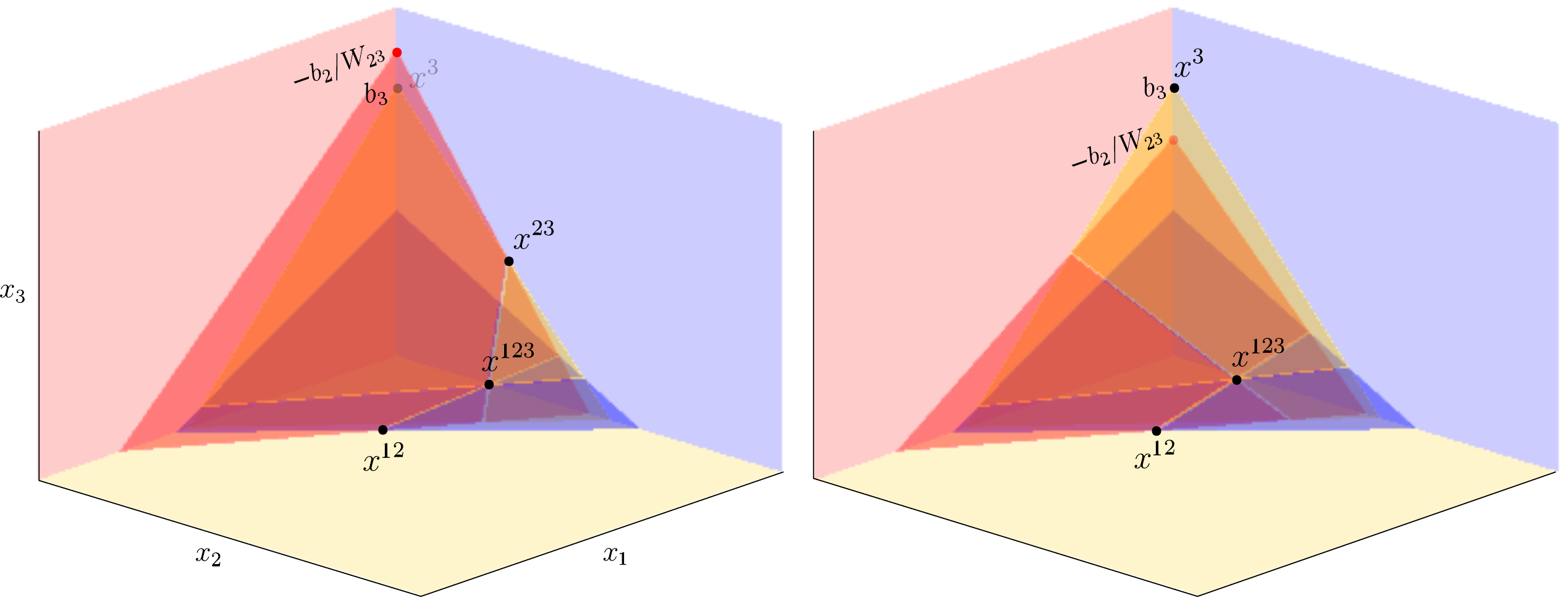}
    \caption{Varying the parameters of the network such that the quantity $b_3W_{23}+b_2$ changes sign changes the relative positions of the hyperplane $H_2$ and $H_3$ along the $x_3$-axis leading to the persistent support bifurcation $\{23,12,123\}\to \{3,12,123\}$. }
    \label{Arrangements 1}
\end{figure} 
 From the arrangement we can read off cocircuits and determine that  $\operatorname{FP}(W,b)=\{12,123,23\}$. To see how modulation of network parameters leads to changes in this set consider flipping the sign of the quantity $b_3W_{23}+b_2$ by letting $W_{23}\to -0.8$. In the arrangement, this pushes the red hyperplane $H_2$ over the virtual fixed point $x^3$. Moreover, as this is done, the fixed point $x^{23}$ collides with $x^3$ and changes position with respect to $E_2$ and we have the persistent support bifurcation $\{23\}\to \{3\}$. If we encode the signs of $b_iW_{ji}+b_j$ in a directed graph, then this bifurcation corresponds to deleting an edge so that neuron three becomes a sink (Figure~\ref{graph_11}).

 \begin{figure}[h]
    \centering
     \includegraphics[width=3in]{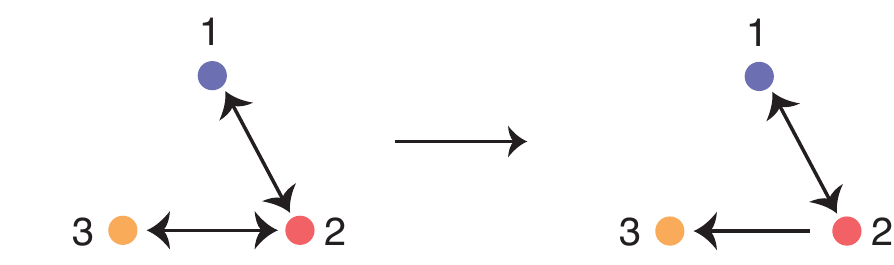}
    \caption{The directed graph of a threshold-linear network is determined by the signs of the quantities $b_iW_{ji}+b_j$. Such a graph captures pairwise asymmetries in the connectivity of the network. The bifurcation in the previous figure corresponds to deleting the edge $3\to 2$ in this graph. }
    \label{graph_11}
\end{figure} 
Although this bifurcation is obtained by letting $W_{23}\to -0.8$, we can similarly consider varying either of the external inputs $b_2$ or $b_3$ in order to flip the sign of $b_3W_{23}+b_2$. Consider for instance, letting $b_2\to 0.25$. While this has the effect of flipping the sign of $b_3W_{23}+b_2$ and realizing the bifurcation $\{23\}\to \{3\}$, varying the external input in this way has a less localized effect on the network and in fact leads to several other bifurcations in addition. Specifically, decreasing the external input to neuron two in this way produces three persistent bifurcations $\{12\}\to \{1\}$, $\{123\}\to \{13\}$ and $\{23\}\to \{3\}$. In sum, the effect is to change the admissible sets from $\operatorname{FP}(W,b)=\{12,123,23\}$ to $\operatorname{FP}(W,b)=\{1,13,3\}$. The reason for this effect is that the parameter $b_2$ is involved in determining the intersection of $H_2$ with each of the coordinate axes. Thus a perturbation of $b_2$ moves each of these vertices simultaneously leading to less localized changes in the arrangement. In general, an understanding of how perturbations of network parameters leads to changes in the collection of admissible fixed points requires an understanding of the local perturbations of the hyperplane arrangement which we will develop in later sections.

\begin{figure}[h]
\centering
\includegraphics[width=6in]{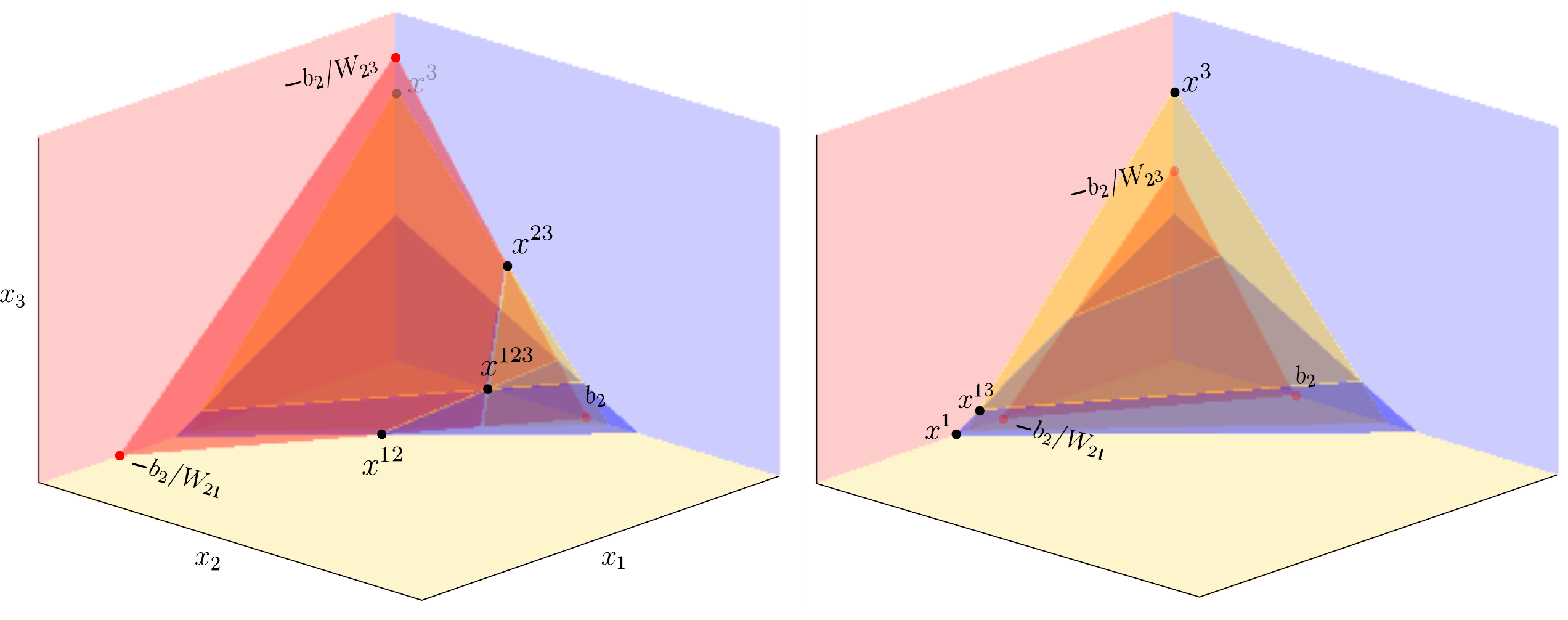}
\caption{Varying the external input leads to multiple support bifurcations. Here, the perturbation $b_2:0.4\to0.25$ contracts the hyperplane $H_2$ along all three coordinate axes simultaneously producing three persistent bifurcations $\{12\}\to \{1\}$, $\{123\}\to \{13\}$ and $\{23\}\to \{3\}$ changing $\operatorname{FP}(W,b)=\{12,123,23\}$ to  $\operatorname{FP}(W,b)=\{1,13,3\}$. }
\label{fig:A_2}
\end{figure}

\par Consider now the quantity $\Delta^{31}_2=b_1W_{32}-b_3W_{12}$. The sign of this quantity determines the position of the intersection of $H_1$ with the $x_2$-axis with respect to the hyperplane $H_3$. In other words, increasing this quantity away from zero increases the separation of the hyperplanes $H_1$ and $H_3$ along the $x_2$-axis. This can be accomplished by increasing $W_{12}$ toward zero and decreasing $W_{32}$ (Figure~\ref{fig: A_3}.) The effect of this is a non-smooth fold bifurcation $\{123,23\}\to \emptyset$ and a change in the admissible sets from $\operatorname{FP}(W,b)=\{12,123,23\}$ to $\operatorname{FP}(W,b)=\{12\}$.
\begin{figure}[h]
\centering
\includegraphics[width=6in]{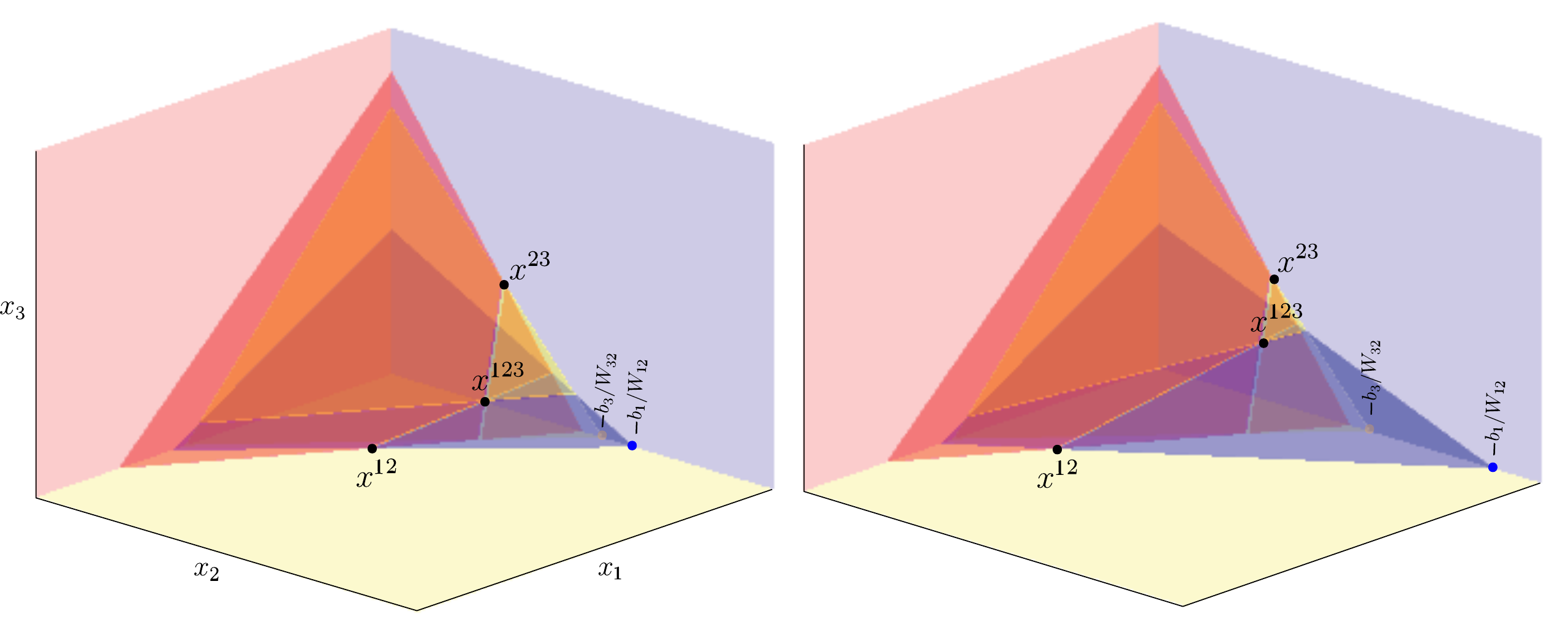}
\caption{The non-smooth fold bifurcation $\{123,23\}\to \emptyset$ can be obtained by varying the quantity $\Delta^{31}_2=b_1W_{32}-b_3W_{12}$ which corresponds to separating the hyperplanes $H_1$ and $H_3$ along the $x_2$ axis. This bifurcation can equivalently be viewed as flipping the simplicial cell bounded by the hyperplanes $\{H_1,H_2,H_3,E_1\}$. We will show in section four that all generic support bifurcations in dimension three correspond to flipping a simplicial cell.  }
\label{fig: A_3}
\end{figure} 
Note that this bifurcation does not change the underlying directed graph as it does not change the position of $H_1$ or $H_3$ with respect to $H_2$ along the $x_2$-axis. Another important observation here is that this bifurcation corresponds to flipping a simplicial cell in the arrangement. If we consider the decomposition of the hyperplane arrangement $\mathcal{A}$ into cells, then an $n$-dimensional cell is called simplicial if it has a minimal number of vertices and faces. The bifurcation in Figure~\ref{fig: A_3} corresponds to flipping the simplicial cell bounded by the hyperplanes $H_1$, $H_2$ and $H_3$ and $E_1$. For an arrangement in general position, the smallest possible changes are in one-to-one correspondence with its simplicial cells. Thus, the set of support bifurcations obtainable by perturbing network parameters is constrained by the set of simplices in the arrangement. By considering the chirotope of the arrangement we will be able to express these constraints algebraically.

\end{example}

%% file: Section-3.tex
\section{The chirotope of a threshold-linear network}
In this section we use the equivalence of cocircuits and chirotopes to express the combinatorial geometry of the previous section in terms of determinants. In later sections we will use constraints on chirotopes in the form of Grassmann-Pl\"ucker relations to obtain constraints on support bifurcations. These relations can be viewed geometrically as constraints on the existence of simplicial cells in the hyperplane arrangement. 

\par We begin by homogenizing the affine hyperplanes in $\mathcal{A}$ by adding a variable $z$ and its corresponding hyperplane $E_\infty:=\{z=0\}$. This gives an arrangement of $2n+1$ hyperplanes in $\mathbb{R}^{n+1}$ which by abuse of notation we will still denote $\mathcal{A}=\{E_1,H_1,...,E_n,H_n,E_\infty\}$. Note that the original affine arrangement is recovered by intersecting the new arrangement with the hyperplane $\{z=1\}$ and vertices of the affine arrangement correspond to lines in the hyperplane arrangement. The coefficients of the equations defining $\mathcal{A}$ define a configuration of vectors  $\{e_1,h_1,...,e_n,h_n,e_\infty\}\subset\mathbb{R}^{n+1}$ which are normal to the hyperplanes in $\mathcal{A}$ and give an equivalent description of the hyperplane arrangement. We will think of these vectors arranged as the rows of a matrix of shape $(2n+1)\times (n+1)$ which we will again refer to as $\mathcal{A}$:
\begin{equation}\label{matrix}
\mathcal{A}:=\small\begin{pmatrix}
1 & 0 & \cdots & 0& 0\\
-1 & W_{12} & \cdots & W_{1n}& b_1\\
 \vdots&\vdots & \ddots &\vdots  &\vdots\\
0 & 0 &\cdots & 1 & 0\\
W_{n1} & W_{n2} & \cdots & -1 & b_n\\
0 & 0 & \cdots & 0 & 1
\end{pmatrix}
\end{equation}
Note that both the threshold linear network and the affine hyperplane arrangement from the previous section are completely specified by the matrix $\mathcal{A}$. 
\begin{definition}
The  \textit{chirotope of a threshold-linear network} we define to be the alternating map:
\begin{align*}
\chi: \mathcal{A}^{n+1}&\longrightarrow \{+,0,-\}\\
(a_1,\dots,a_{n+1})&\longmapsto \text{sign}\:\text{det}(a_1,...,a_{n+1})
\end{align*}
where $a_i$ are vectors corresponding to rows of the matrix $\mathcal{A}$.
\end{definition}

\par The terminology comes from the word $\textit{chiral}$ and the fact that $\chi(a)$ can be viewed as the orientation of the basis defined by the ordered collection of vectors $a\in\mathcal{A}^{n+1}$. A chirotope can be associated to any hyperplane arrangement or equivalently to any configuration of vectors and the above definition is simply a restriction to the case of arrangements arising from threshold-linear networks. In order to express the results of the previous section in the language of chirotopes we define the ordered subset $a^\sigma\in\mathcal{A}^{n+1}$ where the $i$-th vector is given by:
\[
a^\sigma(i):=\begin{cases}
h_i \text{ if } i \in \sigma\\
e_i \text{ if } i \not\in\sigma
\end{cases}
\]
The cocircuit of the point $x^\sigma$ in the affine arrangement is then recovered from the chirotope via the two relations:
\begin{equation}\label{translation}
\begin{split}
C^\sigma(E_i)&=\chi(a^\sigma,e_i) \chi(a^\sigma,e_\infty) \\
C^\sigma(H_i)&=\chi(a^\sigma,h_i) \chi(a^\sigma,e_\infty)
\end{split}
\end{equation}

 This translation between cocircuits and chirotopes is a reflection of their underlying structure as oriented matroids and a proof can be found in \cite{OMbook}. To characterize fixed point supports in terms of $\chi$ we define the subset of determinants:
\begin{equation}\label{s_def}
s^\sigma_i:= \begin{cases}
\text{det}(a^\sigma,e_i) \text{ if } i\in\sigma\cup \infty \\
\text{det}(a^\sigma,h_i) \text{ if } i\not\in\sigma
\end{cases}
\end{equation}
Note then that we have $\text{sgn } s^\sigma_i:= \chi(a^\sigma,e_i)$ for $i\in\sigma\cup \infty$ and $\text{sgn } s^\sigma_i:= \chi(a^\sigma,h_i)$ for $i \not\in \sigma$.  
Also, from the alternating property of determinant, $s^\sigma_i=-s^{\sigma\backslash i}_i$. This notation is borrowed from \cite{fp-paper} where a related collection of determinants $\text{det}((I-W_{\sigma\cup \{i\}})_i;b_{\sigma\cup\{i\}})$ is arrived at via Cramer's rule. The following theorem is a translation of Lemma \ref{the_lemma} and Theorem 2.12 from \cite{fp-paper} in terms of $\chi$.

\begin{theorem}\label{theorem1}
 A subset of neurons $\sigma$ supports a fixed point if and only if for all $i\in\sigma$ and $j\not\in\sigma$:
\[
\text{sgn}\: s^\sigma_i\cdot s^\sigma_\infty=+\hspace{.1cm} \text{ and  }\hspace{.1cm} \text{sgn}\:s^{\sigma}_j\cdot s^\sigma_\infty=-
\]

\end{theorem}
\begin{proof}
This follows directly from Lemma \ref{the_lemma} and \eqref{translation}.
\end{proof}
In the competitive case, it is straightforward to show that $\text{sgn}\:s^\sigma_\infty=\text{sgn}\:s^\sigma_i$ for some $i\in\sigma$ so that $\sigma$ supports a fixed point if and only if:
\begin{equation}\label{fp_condition}
    s^\sigma_i\cdot s^\sigma_j<0
\end{equation}
for all $i\in\sigma$ and $j\not\in\sigma$. Since we are considering only competitive networks here, we will use equation \eqref{fp_condition} whenever computing fixed point supports. 
\par From Theorem \ref{theorem1} and Lemma \ref{lemma2} we see that the quantities $s^{ij}_i=b_jW_{ij}+b_i$ arise from the larger set of determinants $s^\sigma_i$ which by Theorem \ref{theorem1} determine the fixed point supports of the network. Moreover, the quantities $\Delta^{ij}_k$ and $s^\sigma_i$ all arise from the chirotope $\chi$. The advantage of this perspective will come from the structure of $\chi$ as an oriented matroid which we will use in the next section. The following example illustrates Theorem \ref{theorem1} and the characterization of fixed point supports and their bifurcations in terms of determinants.

\begin{example}
Consider again the threshold-linear network in Example \ref{Example1} and whether the subset $\sigma=\{2,3\}$ supports a fixed point. By the above theorem this is equivalent to $\text{sgn }s^{23}_2=\text{sgn }s^{23}_3=-\text{sgn }s^{23}_1$. We can check this condition explicitly by computing:
\begin{align*}
s^{23}_2&=\text{det}(e_1,h_2,h_3,e_2)=\small\begin{vmatrix}
1 &        0     &    0      &   0\\
   -0.65 &  -1&  -0.57  &  0.40\\
   -1.34  & -1.45 &  -1  &  0.62\\
         0 &   1    &   0       &  0
\end{vmatrix}= 0.0466\\
s^{23}_3&=\text{det}(e_1,h_2,h_3,e_3)=\small\begin{vmatrix}
1 &        0     &    0      &   0\\
   -0.65 &  -1&  -0.57  &  0.40\\
   -1.34  & -1.45 &  -1  &  0.62\\
         0 &   0   &  1       &  0
\end{vmatrix}=0.04\\
s^{23}_1&=\text{det}(e_1,h_2,h_3,h_1)=\small\begin{vmatrix}
1 &        0     &    0      &   0\\
   -0.65 &  -1&  -0.57  &  0.40\\
   -1.34  & -1.45 &  -1  &  0.62\\
  -1&  -0.97&   -1.47&    0.49
\end{vmatrix}=-0.019
\end{align*}
Thus, we have $\text{sgn }s^{23}_2=\text{sgn }s^{23}_3=-\text{sgn }s^{23}_1$ and $\{2,3\}$ supports a fixed point i.e. $x^{23}$ is admissible. To understand how the support bifurcation $\{123,23\}\to \emptyset$ exhibited in Example \ref{Example1} arises from this perspective, note that the determinant $s^{23}_1=-s^{123}_1$ is involved in the admissibility of both $x^{23}$ and $x^{123}$. Now, if we view the determinants defining $\chi$ as polynomials in the network parameters, then they induce a cell decomposition such that $\chi$ is constant in each cell. For instance, the determinant $s^{123}_1=\text{det}(h_1,h_2,h_3,e_1)$ corresponds to the degree three polynomial:
\[
s^{123}_1=W_{23}W_{32}b_1 - W_{12}b_2 - W_{13}b_3 - W_{12}W_{23}b_3 - W_{13}W_{32}b_2 - b_1
\]
and the zero set $s^{123}_1=0$ defines a polynomial hypersurface in the space of threshold-linear networks partitioning it according to the sign of the polynomial function $s^{123}_1$.

\begin{figure}[h]
    \centering
    \includegraphics[width=5in]{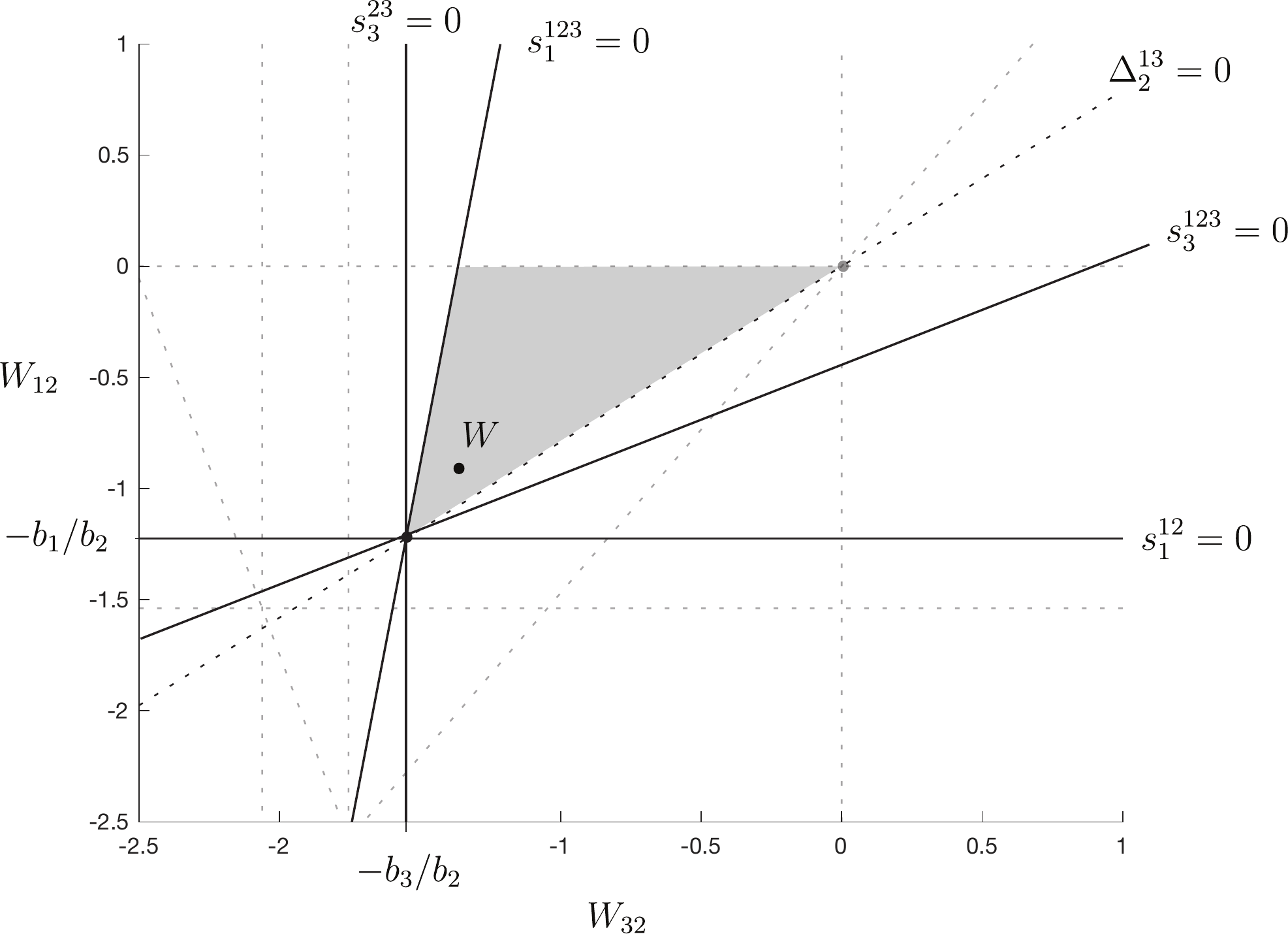}
    \caption{Two dimensional slice of the cell decomposition of the parameter space induced by $\chi$. The network in Example \ref{Example1} lies in the cell of this decomposition shaded gray.  The bifurcation in Figure~\ref{fig: A_3} corresponds to a wall of this cell and is obtained by varying the network parameters  across it. This can be viewed as varying the parameters orthogonally to the line $\Delta^{13}_2=0$ and separating the hyperplanes $H_1$ and $H_3$ along the $x_2$ axis. }
    \label{fig:slice}
\end{figure}
The support bifurcation $\{123,23\}\to \emptyset$ corresponds to varying the network parameters across this hypersurface. To make this more concrete, consider the two-dimensional slice of the parameter space obtained by fixing all parameters except $W_{12}$ and $W_{32}$. The cell decomposition restricted to this slice is an arrangement of lines in the plane (Figure~\ref{fig:slice}.) The lines $s^{12}_1=0$ and $s^{23}_3=0$ correspond to boundaries between directed graphs. The support bifurcation $\{23,123\}\to \emptyset$  in Example \ref{Example1} arises because $s^{123}_1=0$ defines a wall of the cell containing the network. The bifurcation is realized by varying the parameters across this wall, changing the sign of the determinant $s^{123}_1$. Geometrically, the sign of the determinant $s^{123}_1$ corresponds to the orientation of the configuration of vectors $\{h_1,h_2,h_3,e_1\}$ and varying the parameters across the hypersurface $s^{123}_1=0$ corresponds to flipping this orientation. Equivalently, the hyperplanes $\{H_1,H_2,H_3,E_1\}$ bound a simplicial cell in the arrangement and varying the parameters across $s^{123}_1=0$ pushes a face of this simplex across the opposite vertex and flips the simplex. To realize this bifurcation, in Example \ref{Example1} we varied the quantity $\Delta^{13}_2=b_3W_{12}-b_1W_{32}$ by increasing $W_{12}$ toward zero and decreasing $W_{32}$. Here, we view this is as moving orthogonally to the hypersurface $\Delta^{13}_2=0$ in the cell decomposition.
\end{example}

\begin{remark}
For competitive networks, the determinants $s^\sigma_\infty$ do not play a role in support bifurcations. However, in the non-competitive case, the hypersurfaces $s^\sigma_\infty=0$ give rise to support bifurcations of the form $\sigma\to\emptyset$. In such a bifurcation, the fixed point $x^\sigma$ goes off to infinity and disappears. If we view the arrangement $\mathcal{A}$ as an arrangement of spheres on $S^n$, then such a bifurcation can be viewed as the fixed point $x^\sigma$ crossing the equator $z=0$ and then reappearing as a virtual fixed point at the antipodal point on the equator.
\end{remark}

%% file: Section-4.tex
\section{Mutations of chirotopes and support bifurcations}
\par In this section we show how to obtain constraints on admissible sets of fixed point supports and their bifurcations by characterizing walls in the cell decomposition induced by $\chi$ as mutations of $\chi$. Roughly speaking, mutations in our context correspond to sign flips that result in another chirotope of a threshold-linear network.  Constraints on mutations are then obtained by considering the structure of determinant maps in general. Many of the results and definitions in this section are adapted from \cite{piano-movers} and \cite{Sturmfels-Mutations}. In particular, a similar approach was used in \cite{piano-movers} to obtain constraints in the Piano Mover's problem.
\par A chirotope $\chi$ is called \textit{simplicial} if $\chi:\mathcal{A}^{n+1}\to \{+,-\}$. In other words, $\chi(\lambda)\not=0$ for any $\lambda\in\mathcal{A}^{n+1}$. An ordered basis $\lambda\in\mathcal{A}^{n+1}$ such that $\chi(\lambda)=0$ corresponds to the affine hyperplanes corresponding to $\lambda$ intersecting at a point. The perturbations of such arrangements are not in one to one correspondence with simplicial cells. For instance, a perturbation of the degenerate arrangement in Figure~\ref{fig:degenerate} flips several simplices at once and leads to a higher order support bifurcation.
\begin{figure}[h]
    \centering
    \includegraphics[width=6in]{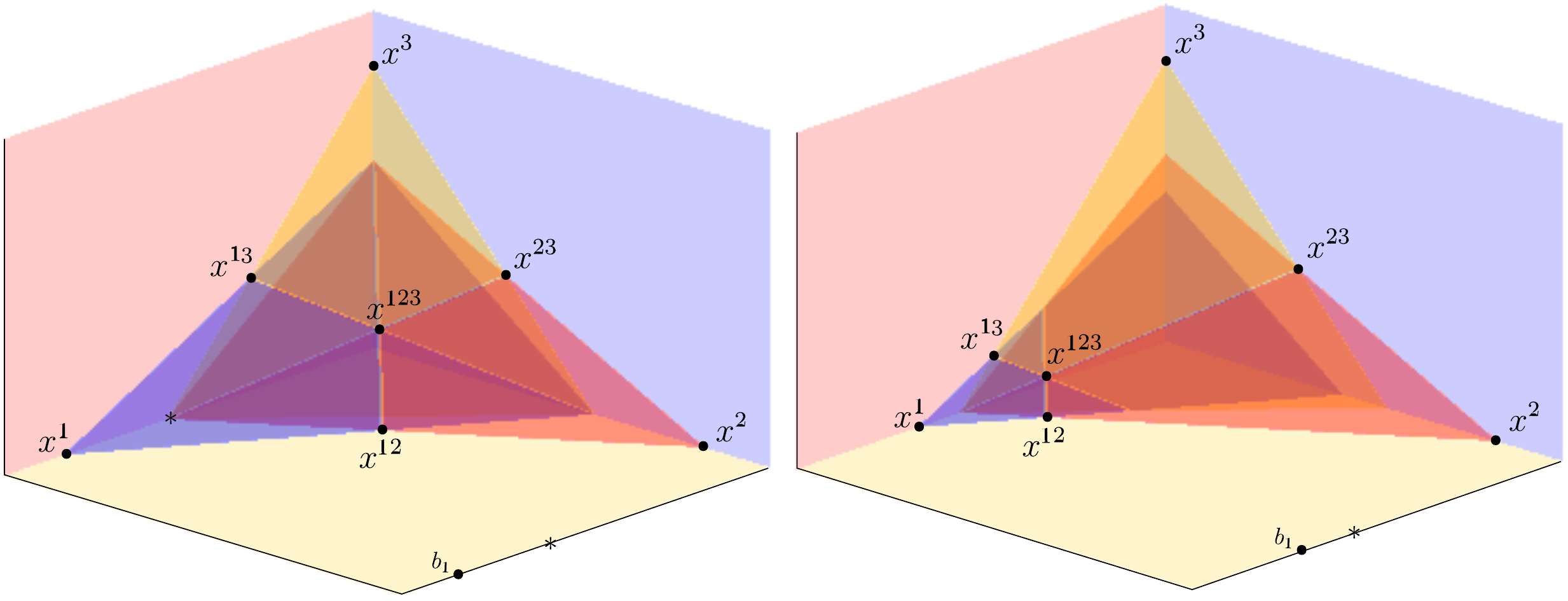}
    \caption{ The degeneracy in this arrangement leads to the generic support bifurcation $\{123,13,12,1\}\to \emptyset$ obtained by pushing the hyperplane $H_1$ over the degeneracy. The corresponding chirotope in this situation is not simplicial and the bifurcation does not correspond to a mutation.}
    \label{fig:degenerate}
\end{figure}

In terms of $\chi$, in such a scenario several determinants have become coupled so that flipping the sign of one determinant forces the signs of other determinants to be flipped as well. Such dependencies are captured by the Grassmann-Pl\"ucker relations. These are quadratic relations in the determinants arising from the Pl\"ucker embedding of the Grassmannian manifold into projective space via the determinant map. In the simplicial case, these relations can be reduced to the following \textit{three-term} Grassman-Pl\"ucker relations:

\begin{definition}\label{GP_def}
Given $\sigma\in\mathcal{A}^{n-1}$ and $\tau\in\mathcal{A}^4$ the relation:
\begin{align*}
    &\text{det}(\sigma,\tau_1,\tau_2)\cdot \text{det}(\sigma,\tau_3,\tau_4)- \\
    &\text{det}(\sigma,\tau_1,\tau_3)\cdot\text{det}(\sigma,\tau_2,\tau_4)+ \\
    &\text{det}(\sigma,\tau_1,\tau_4)\cdot\text{det}(\sigma,\tau_2,\tau_3)=0 
\end{align*}
is called a \textit{three-term Grassmann-Pl\"ucker relation}.
\end{definition}

In our case, a three-term Grassmann-Pl\"ucker relation is a polynomial equality in the parameters of the network and thus imposes strong constraints on the geometry of the cell decomposition of the space of threshold-linear networks induced by $\chi$.

\begin{example}
 Consider the three-term  Grassmann-Pl\"ucker relation defined by $\sigma=\{e_2,h_2\}$ and $\tau=\{e_1,e_3,h_1,h_3\}$:
\begin{align*}
    &\text{det}(e_2,h_2,e_1,e_3)\cdot \text{det}(e_2,h_2,h_1,h_3)- \\
    &\text{det}(e_2,h_2,e_1,h_1)\cdot\text{det}(e_2,h_2,e_3,h_3)+ \\
    &\text{det}(e_2,h_2,e_1,h_3)\cdot\text{det}(e_2,h_2,e_3,h_1)=0 
\end{align*}
From the matrix in \eqref{matrix} we see that this is the polynomial equality:
\begin{align*}
&-b_2( 
W_{13}W_{31}b_2 - W_{21}b_1 - W_{23}b_3 - W_{13}W_{21}b_3 - b_2 - W_{23}W_{31}b_1)\\
&-(W_{23}b_1 - W_{13}b_2)(W_{31}b_2 - W_{21}b_3)\\
&-(b_2 + W_{23}b_3)(b_2 + W_{21}b_1)=0
\end{align*}
Equivalently, this is can be written:
\[
 -s_2^2 s^{123}_1-\Delta^{21}_3\Delta^{32}_1-s^{23}_2 s^{12}_2=0
\]
In this way, the Grassmann-Pl\"ucker relations identify relationships between the parameters of the network that may not be immediately apparent. Note, if any two of the three terms in this relation have the same sign, then the sign of the third term is uniquely determined. Thus, these relations impose constraints on how the map $\chi$ can change as parameters are varied.
\end{example}

\begin{definition}
For any simplicial chirotope $\chi:\mathcal{A}^{n+1}\to\{+,-\}$, $\lambda\in\mathcal{A}^{n+1}$ is said to be a \textit{mutation} of $\chi$ if the alternating map $\chi'$ obtained from $\chi$ by flipping the sign $\chi(\lambda)$ satisfies the three-term Grassmann-Pl\"ucker relations.

\end{definition}
Geometrically, a mutation corresponds to the existence of a simplicial cell in the arrangement \cite{Sturmfels-Mutations}. In this way, the local changes to the arrangement are in one-to-one correspondence with the simplicial cells of the arrangement. Note however, if $\chi$ is the chirotope of a threshold-linear network and $\lambda\in\mathcal{A}^{n+1}$ is a mutation of $\chi$, the chirotope obtained by flipping the sign $\chi(\lambda)$ is not necessarily the chirotope of a threshold-linear network. Mutations that can be obtained by varying parameters in the space of threshold-linear networks we will call realizable. Realizable mutations correspond to walls in the cell decomposition of the parameter space induced by $\chi$. The following lemma gives the converse for dimension three.

\begin{lemma}\label{determinants_lemma}
In dimension three, a wall in the cell decomposition of the space of threshold-linear networks induced by $\chi$ corresponds to a mutation of $\chi$.
\end{lemma}
\begin{proof}
We would like to show that no determinant is identically zero and that no two determinants have a common factor. The first condition implies that a generic point in the parameter space corresponds to a simplicial chirotope. The second condition implies that the chirotopes of two adjacent cells differ by the sign of a single determinant. We will prove this by explicitly computing the determinants defining $\chi$ in dimension three. Let $p$ be the permutation parity of $(i,j,k)$ then up to relabeling of neurons, the ${{2\cdot 3+1}\choose{3+1}}=35$ determinants are:
\begin{align*}
    s^\emptyset_\infty&=1\\
    s^i_\infty&=-1\\
    \text{det}(e_i,h_i,e_k,e_\infty)&=(-1)^pW_{ji}\\
    s^i_i&=-b_i\\
    \Delta^{ij}_k&=(-1)^p (b_jW_{ik}-b_iW_{jk})\\
    s^{ij}_j&=b_j+b_iW_{ji}\\
    s^{ij}_\infty&= 1-W_{ij}W_{ji}\\
    \text{det}(h_i,h_j,e_j,e_\infty)&=(-1)^p( W_{jk}+W_{ik}W_{ji})\\
    s^{123}_\infty&=W_{12}W_{21} + W_{13}W_{31} + W_{23}W_{32}\\
    &+ W_{12}W_{23}W_{31} + W_{13}W_{21}W_{32} - 1\\\
    s^{ijk}_i & =- W_{ij}(b_kW_{jk}+b_j)-b_kW_{ik}-b_i\\
    &-W_{kj}(b_jW_{ik}-b_iW_{jk})
\end{align*}
We see that no determinant is identically zero and so the chirotope of a generic threshold-linear network is simplicial. It remains to show that no two determinants have a common factor. In fact, every determinant is distinct and irreducible. To see this, we will present a proof for $s^{ijk}_i$. Similar arguments apply to the remaining determinants. Let $p$ be the permutation parity of $(i,j,k)$ then we have:
\begin{align*}
s^{ijk}_i&=(-1)^p\text{det}(h_i,h_j,h_k,e_i)=-\begin{vmatrix}
1 & 0 & 0 & 0\\
0 & W_{ij} & W_{ik} & b_i\\
0 & -1 & W_{jk} & b_j\\
0 & W_{kj} & -1 & b_k
\end{vmatrix}\\
\\
&=-W_{ij}(b_kW_{jk}+b_j)-b_kW_{ik}-b_i\\
    &-W_{kj}(b_jW_{ik}-b_iW_{jk})
\end{align*}
By inspection we see that $s^{ijk}_i$ is degree one in each of the variables $W_{ij}$, $W_{ik}$ ,$b_i$, $W_{jk}$, $b_j$ ,$W_{kj}$ and $b_k$. Suppose this polynomial was reducible so that $s^{ijk}_i=f\cdot g$ for some polynomials $f$ and $g$.  It follows that for each variable, $\text{deg}(f)+\text{deg}(g)=1$. Consider the variable $W_{ij}$. Without loss of generality assume that $f$ is degree one in $W_{ij}$ and $g$ is degree zero. By the definition of determinant, $f$ must be degree one in the variables $W_{ik}$, $b_i$, and $W_{kj}$ since these variables lie in either row $i$ or row $j$. Continuing this reasoning, the polynomial $f$ must be degree one in every variable and $g$ must be degree zero. Thus, the factorization $s^{ijk}_i=f\cdot g$ is trivial and $s^{ijk}_i$ must be irreducible.
\end{proof}

 The lemma implies that, in dimension three, any generic support bifurcation of a threshold-linear network corresponds to a fixed point crossing a single linear boundary and is of the form $\{\sigma,\sigma\cup i\}\to \emptyset$ or $\{\sigma\}\to \sigma\cup i$. In other words, the scenario in Figure~\ref{fig:degenerate} does not arise generically in the space of three-dimensional threshold-linear networks. Moreover, since every support bifurcation corresponds to a mutation of $\chi$ it must preserve the three-term Grassmann-Pl\"ucker relations. We will use this fact to obtain constraints on support bifurcations of threshold-linear networks. To capture the constraints imposed by a graph $G$ on mutations and support bifurcations we make the following definition.
\begin{definition}
The dual graph of the cell decomposition induced by $\chi$ on the space of threshold-linear networks with graph $G$ we will call the \textit{mutation graph} of $G$. The dual graph of the cell decomposition induced by fixed point supports we will call the \textit{bifurcation graph of $G$}.
\end{definition}

This definition is motivated by \cite{Sturmfels-Mutations} and the mutation graph $\mathcal{G}^{k,n}_{real}$ of all chirotopes arising from arrangements of $k$ hyperplanes in $\mathbb{R}^n$. The following lemma is an analogue of a result from  \cite{Sturmfels-Mutations} stating that $\mathcal{G}^{k,r}_{real}$ is connected. Connectedness in our context implies that any two threshold-linear networks are related by a sequence of order one support bifurcations.

\begin{lemma}
The mutation and bifurcation graphs of a directed graph $G$ on three nodes are connected.
\end{lemma}
\begin{proof}
It is sufficient to show that the mutation graph of $G$ is connected since the bifurcation graph is a contraction of the mutation graph. Since the space of threshold-linear networks with graph $G$ is  is a convex open subset of $\mathbb{R}^{n^2-n}$, there is a $n^2-n$ dimensional family of lines between a point in one cell and the points contained in another cell of the decomposition induced by $\chi$. We must show that there exists a line in this family that does not pass through any intersection of two determinants. By the lemma above any such intersection has codimension two and so the family of lines that pass through an intersection has codimension at least two. Thus the generic line between two cells will not pass through an intersection.
\end{proof}

\begin{figure}[h]
    \centering
    \includegraphics[scale=.7]{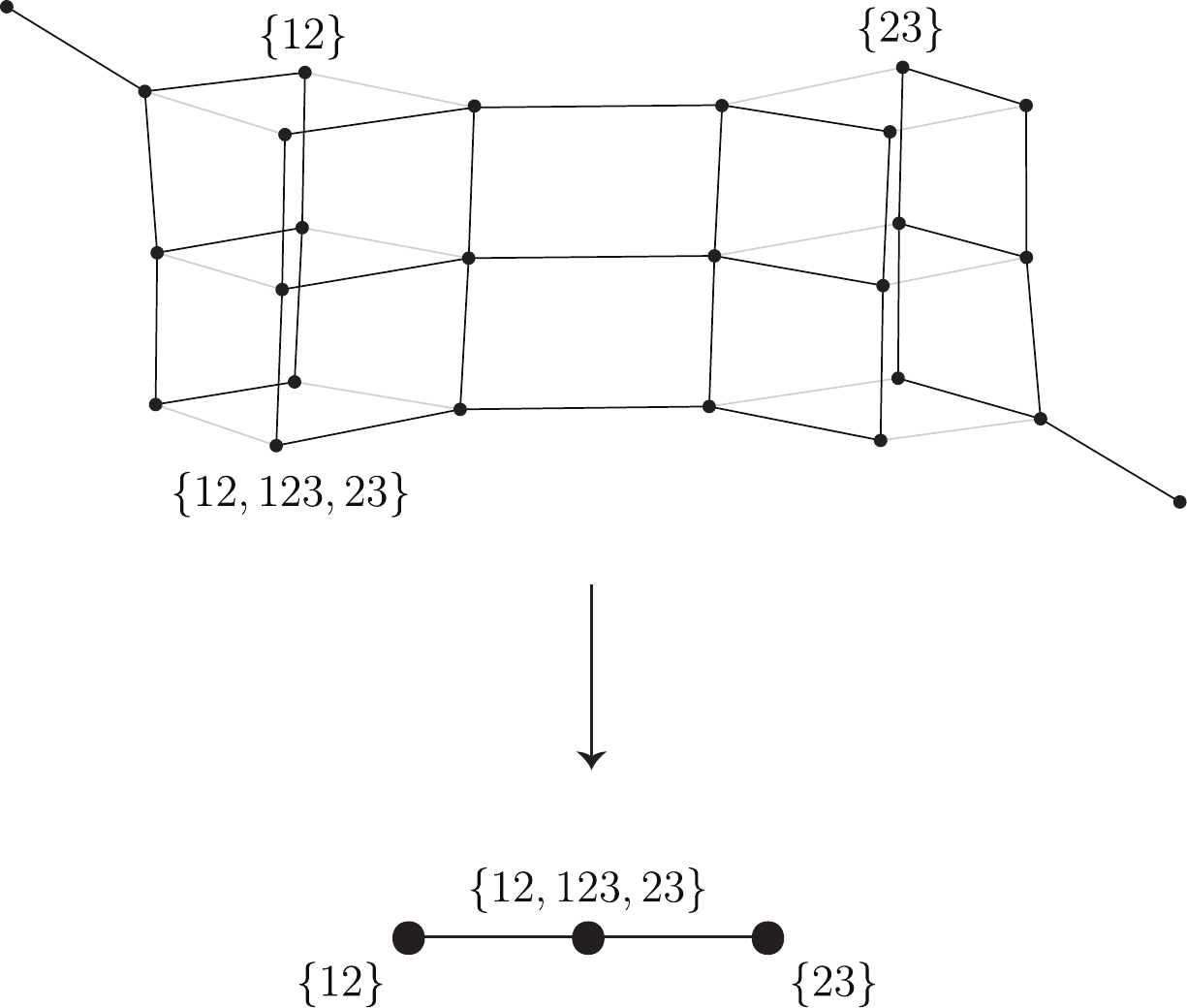}
    \caption{The mutation and bifurcation graph of the first graph in Figure~\ref{graph_11}. The bifurcation graph is obtained from the mutation graph by contracting mutations that do not correspond to support bifurcations.}
    \label{fig:mutation_graph}
\end{figure}

To end this section we outline a procedure for obtaining constraints on bifurcation graphs using a result of \cite{Sturmfels-Mutations} providing an efficient way to obtain constraints on the set of mutations of a chirotope. Let $\lambda=(\lambda_1,...,\lambda_{n+1})\in\mathcal{A}^{n+1}$ be an ordered basis and $\mu=(\mu_1,...,\mu_n)\in\mathcal{A}^{n}$ the complement of $\lambda$ in $\mathcal{A}$. For $i=1,...,n+1$ and $j=1,...,n$ define:
\[
\lambda[i\to j]:=(\lambda_1,...,\lambda_{i-1},\mu_j,\lambda_{i+1},...,\lambda_{n+1})
\]
The standard representative matrix of $\chi$ with respect to $\lambda$ is the $(n+1)\times n$ matrix:
\[
T[\lambda]:=\bigg( \chi(\lambda[i\to j])\bigg)
\]
From \cite{Sturmfels-Mutations} we have the following characterization of mutations in terms of this matrix.
\begin{theorem}\label{sturmfels}(Roudneff-Sturmfels, 1988)
$\lambda\in\mathcal{A}^{n+1}$ is a mutation if and only if $T[\lambda]$ has rank one.
\end{theorem}
This result implies strong constraints on mutation graphs. In our present case, we wish to impose further constraints corresponding to our choice to study competitive threshold-linear networks with positive external drive and fixed graph $G$. From the proof of Lemma \ref{determinants_lemma} we see  that the quantities $b_i$, $W_{ij}$ and $b_iW_{ji}+b_i$ all arise from the determinant map $\chi$. Thus, our restriction to competitive threshold-linear networks with positive external input and fixed directed graph can be viewed as a set of constraints on the set of mutations of the chirotope. This suggests the following steps for deriving a graph $B$ that contains the mutation graph of $G$ as a subgraph.
\begin{enumerate}
    \item Choose an arbitrary competitive threshold-linear network with positive external input and graph $G$  and compute its chirotope $\chi_0$. Set $B=\{\chi_0\}$.
    \item Compute the set of mutations $\text{Mut}(\chi_0)$ using Theorem \ref{sturmfels}
    \item Compute the neighbors of $\chi_0$ by flipping signs in $\chi_0$ corresponding to the set $\text{Mut}(\chi_0)\backslash \text{Constraints}$ and append them to $B=\{\chi_0\}$.
    
\end{enumerate}

This procedure produces a connected graph $B$ with the mutation graph of $G$ as a connected subgraph. Moreover, the mutation graph is a proper subgraph of $B$ if and only if there exist mutations that are not realizable. By contracting $B$ along edges that do not change the fixed point supports we obtain a graph that contains the bifurcation graph of $G$ as a subgraph (Figure~\ref{fig:mutation_graph}). In the next section we will use this procedure to obtain the bifurcation graph of every directed graph in dimension three.

%% file: Section-5.tex
\section{Dimension three bifurcation graphs }

In this section we exhibit every admissible collection of fixed point supports and every support bifurcation of a three-dimensional competitive threshold-linear network with positive external input. We use the procedure outlined in the previous section to construct bifurcation graphs for each of the sixteen non-isomorphic directed graphs on three nodes and prove all bifurcations obtained this way are in fact realizable. Moreover, we show how the support bifurcations constrained by a graph $G$ can be realized by modulation of the quantities $\Delta^{ij}_k=b_jW_{ik}-b_iW_{jk}$.

\begin{figure}[h]
    \centering
    \includegraphics[width=6in]{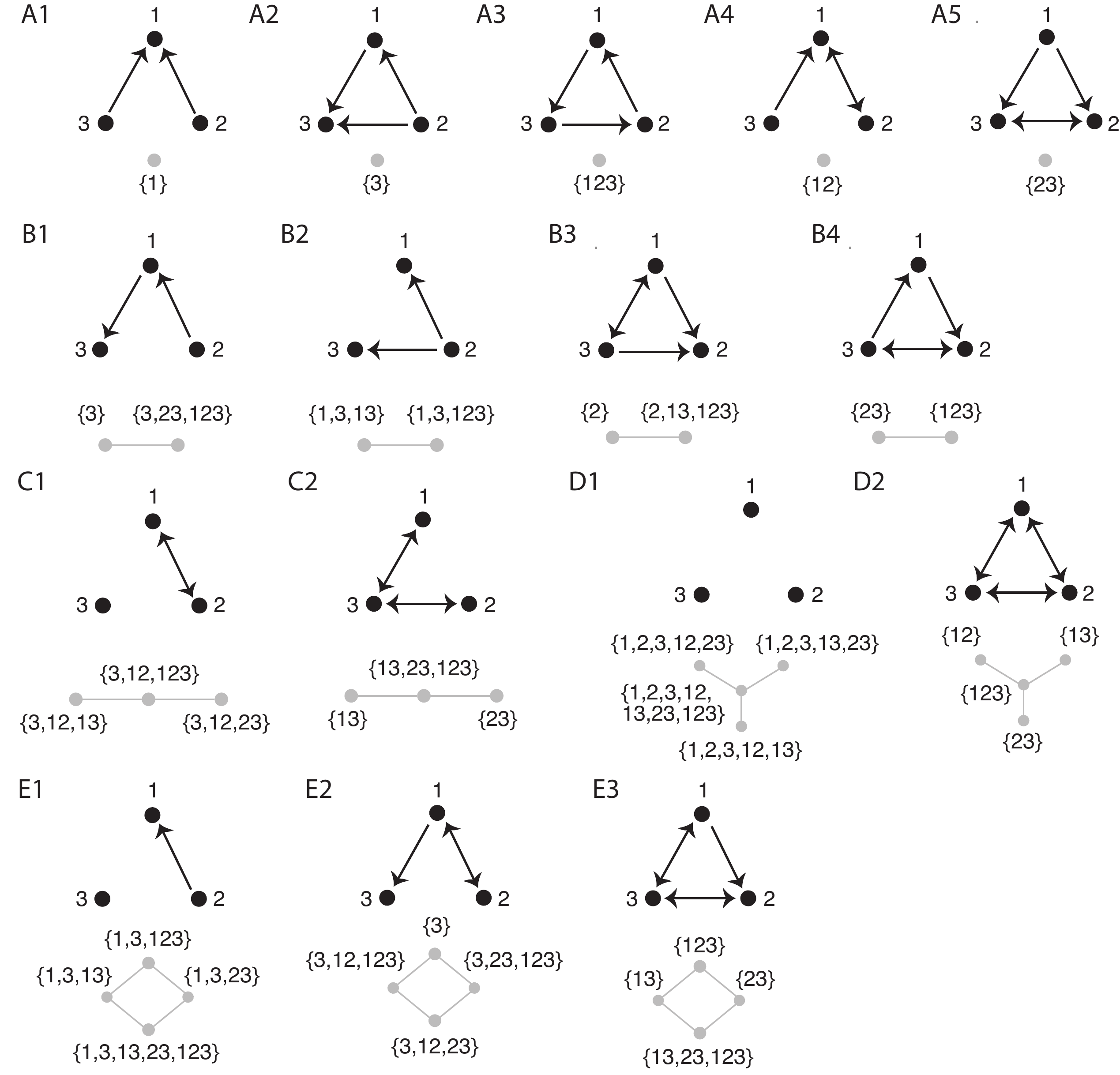}
    \caption{Fixed point supports and bifurcations for competitive threshold-linear networks with positive external drive partitioned according to the sixteen non-isomorphic directed graphs on three nodes. }
    \label{fig:bifs}
\end{figure}

We begin by showing explicitly how the quantities $\Delta^{ij}_k$ relate to the chirotope.

\begin{lemma}\label{separation_determinant}
Let $(i,j,k)$ be an ordered triple of neurons and $p$ the permutation parity of this ordered triple. Then we have:
\[
\Delta^{ij}_k:=(-1)^p\text{det}(e_i,e_j,h_i,h_j)
\]
\end{lemma}
\begin{proof}
Let $p$ be the permutation parity of $(i,j,k)$. By permuting columns and row reducing we have:
\begin{align*}
\text{det}(e_i,e_j,h_i,h_j)=(-1)^p\cdot\small\begin{vmatrix}
1 & 0 & 0 & 0\\
0 &1 & 0 & 0\\
-1 & W_{ij} & W_{ik} & b_i\\
W_{ji} & -1 & W_{jk} & b_j
\end{vmatrix}&=(-1)^p\cdot\small\begin{vmatrix}
1 & 0 & 0 & 0\\
0 &1 & 0 & 0\\
0 & 0 & W_{ik} & b_i\\
0 & 0 & W_{jk} & b_j
\end{vmatrix}\\
\\
&=(-1)^p(b_jW_{ik}-b_iW_{jk})
\end{align*}
\end{proof}
Using this lemma and the Grassmann-Pl\"ucker relations we can also relate the signs of the quantities $\Delta^{ij}_k$ with the directed graph determined by the signs of the quantities $s^{ij}_j$.
\begin{lemma}\label{graph_sep_lemma}
If $k\to i$ but $k\not\to j$ in $G$, then we have $\text{sgn}\:\Delta^{ij}_k=+$. Similarly if $k\to j$ but $k\not\to i$ then we have $\text{sgn}\:\Delta^{ij}_k=-$.
\end{lemma}
\begin{proof}
Consider the Grassmann-Pl\"ucker relation corresponding to $\sigma=\{e_i,e_j\}$ and $\tau=\{e_k,h_i,h_j,h_k\}$:
\begin{align*}
    &\text{det}(e_i,e_j,e_k,h_i)\cdot \text{det}(e_i,e_j,h_j,h_k)- \\
    &\text{det}(e_i,e_j,e_k,h_j)\cdot\text{det}(e_i,e_j,h_i,h_k)+ \\
    &\text{det}(e_i,e_j,e_k,h_k)\cdot\text{det}(e_i,e_j,h_i,h_j)=0 \\
\end{align*}
By transposing vectors and using the alternating property of the determinant we can rewrite this relation:
\begin{align*}
    -&\text{det}(h_i,e_j,e_k,e_i)\cdot \text{det}(e_i,h_j,h_k,e_j)- \\
    -&\text{det}(e_i,h_j,e_k,e_j)\cdot\text{det}(h_i,e_j,h_k,e_i)+ \\
    -&\text{det}(e_i,e_j,h_k,e_k)\cdot\text{det}(e_i,e_j,h_i,h_j)=0 
\end{align*}
Now, if we let $p$ be the permutation parity of $(i,j,k)$, then by definition as well as the previous lemma, this relation becomes:
\[
-(-1)^p s^i_i (-1)^p s^{jk}_j+(-1)^p s^j_j(-1)^p s^{ik}_i-(-1)^p s^k_k (-1)^p\Delta^{ij}_k=0
\]
or more simply,
\[
- s^i_i  s^{jk}_j+ s^j_j s^{ik}_i- s^k_k \Delta^{ij}_k=0
\]
Now, suppose $k\to i$ but $k\not\to j$. Then together with the assumption of positive external input we have $-s^i_is^{jk}_j<0$ and $s^j_js^{ik}_i<0$ and the above relation implies $\Delta^{ij}_k>0$. Similarly, if $k\not\to i$ but $k\to j$, then we have $-s^i_is^{jk}_j>0$ and $s^j_js^{ik}_i>0$ and the above relation implies $\Delta^{ij}_k<0$.
\end{proof}

These two lemmas suggest the following definition:
\begin{definition}\label{sep_def}
We say that neuron $k$ separates node $i$ from node $j$ if $k\to i$ but $k\not\to j$ in $G$.
\end{definition}
Lemma \ref{graph_sep_lemma} says that when neuron $k$ separates neurons $i$ and $j$ in the graph $G$ that the sign of the separation quantity $\Delta^{ij}_k$ is fixed. Alternatively, lemma \ref{separation_determinant} says that when neuron $k$ separates neurons $i$ and $j$, the relative positions of the hyperplanes $H_i$ and $H_j$ along the $x_k$-axis remains fixed. More generally, the quantity $\Delta^{ij}_k$ can be viewed as modulating the separation of these hyperplanes along the $x_k$-axis. With the above lemmas and definition we have the following theorem which gives a graph theoretic condition for the existence of a support bifurcation in dimension three and moreover, confirms the existence of all bifurcations in Figure~\ref{fig:bifs}.
\begin{theorem}\label{theorem2}
In dimension three, a support bifurcation of the form $\{ijk,jk\}\to\emptyset$ or $\{ijk\}\to \{jk\}$ can occur  if and only if one of the following hold in $G$:
\begin{enumerate}
    \item $j\leftrightarrow k$ and $j$ or $k$ is non-separating.
    \item $j\not\leftrightarrow k$ and $j$ or $k$ is non-separating.
\end{enumerate}
\end{theorem}
\begin{proof}
Such a bifurcation occurs if $\{j,k\}$ suppports a fixed point and $s^{ijk}_i$ arises as a realizable mutation. The first condition is equivalent to:  
\[
\text{sgn }\:s^{jk}_j=\text{sgn }\:s^{jk}_k=-\text{sgn }\:s^{jk}_i
\]
where the first equality is guaranteed if $j\leftrightarrow k$ or $j\not\leftrightarrow k$. Consider the three-term  Grassmann-Pl\"ucker relation defined by $\sigma=\{e_i,h_i\}$ and $\tau=\{e_j,e_k,h_j,h_k\}$. After simplification we have:

\begin{equation}\label{GP_Relation}
 s^{ijk}_i+\Delta^{ij}_k \Delta^{ik}_j+s^{ik}_i s^{ij}_i=0
\end{equation}

Suppose $k$ and $j$ are both separating, in such case we have the equalities $\text{sgn }\: \Delta^{ij}_k=\text{sgn } s^{ik}_i$ and $\text{sgn }\Delta^{ik}_j=\text{sgn }s^{ij}_i$ and it follows that:

\begin{align*}
\text{sgn }\: \Delta^{ij}_k\Delta^{ik}_j=\text{sgn }\:s^{ik}_i s^{ij}_i
\end{align*}

Thus, it follows that $\text{sgn }\:s^{ijk}_i$ is determined by the relation \eqref{GP_Relation} and the sign of $s^{ijk}_i$ cannot be flipped by varying network parameters. We now show that in the absence of such a graph structure, the sign of the determinant $s^{ijk}_i$ can flipped. Suppose that $j$ and $k$ are not both separating. Without loss of generality  suppose $k$ is not separating. Consider the family of two-dimensional slices of the parameter space obtained by fixing all parameters except $W_{ik}$ and $W_{jk}$. The determinants $s^{ik}_i=b_kW_{ik}+b_i$, $s^{jk}_j=b_kW_{jk}+b_j$ and $\Delta^{ij}_k=b_jW_{ik}-b_iW_{jk}$ all intersect at $(-b_i/b_k,-b_j/b_k)$ in this slice and the intersection of those threshold-linear networks with non-separating $k$ with this slice corresponds to the positive or negative quadrant.

\begin{figure}[h]
    \centering
    \includegraphics[width=4.5in]{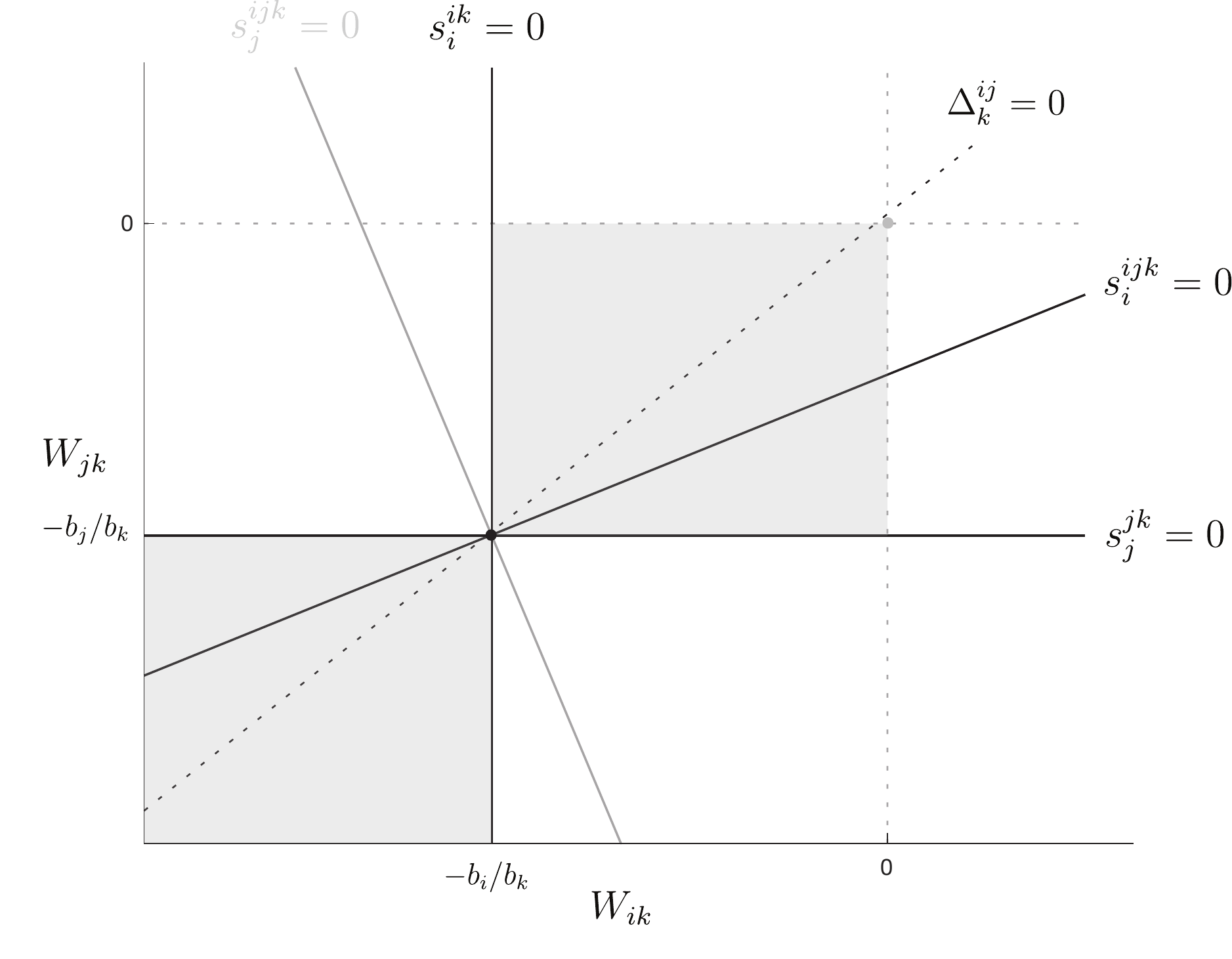}
    \caption{The sign of the determinant $s^{ijk}_i$ can be flipped by varying network parameters if the line $s^{ijk}_i=0$ intersects the quadrant defined by $G$, shaded in gray. In such a case, increasing the separation $\Delta^{ij}_k$ by moving orthogonally to line $\Delta^{ij}_k=0$,  the sign of $s^{ijk}_i$ can be flipped. }
    \label{fig:slice2}
\end{figure}

Consider the relation $\sigma=\{e_i,h_k\}$ and $\tau=\{e_j,e_k,h_i,h_j\}$:

\begin{equation}\label{GP2}
-s^{ijk}_i-s^{ik}_i s^{jk}_k-s^{jk}_j\Delta^{ik}_j=0
\end{equation}
This relation implies that if $s^{ik}_i=0$ and $s^{jk}_j=0$ then $s^{ijk}_i=0$ as well.  Thus, $\Delta^{ij}_k=0$, $s^{ik}_i=0$, $s^{jk}_j=0$ and $s^{ijk}_i=0$ define a central arrangement of lines centered at $(-b_i/b_k,-b_j/b_k)$ (Figure~\ref{fig:slice2}). We would like to show that the line $s^{ijk}_i=0$ intersects the positive and negative quadrants for some slice.  Let $w$ be a point in the parameter space such that $s^{ijk}_i(w)=0$. Equation \eqref{GP2} implies

\[
\text{sgn }\:s^{ik}_i(w) s^{jk}_j(w) =-\text{sgn }\:s^{jk}_k(w) \Delta^{ik}_j(w)
\]
so that the line $s^{ijk}_i=0$ intersects the positive and negative quadrants if and only if 
\begin{equation}\label{caveat}
\text{sgn }\: s^{jk}_k(w) \Delta^{ik}_j(w)=-
\end{equation}
If $j$ happens to separate $i$ and $k$, then this is holds. On the other hand, if $j$ is non-separating, we are then free to vary the parameter $\Delta^{ik}_j$ until the condition again holds. Note that this can be viewed as varying the two-dimensional slice which we are considering. Thus there exists a slice for which the line $s^{ijk}_i=0$ intersects the quadrant defined by $G$. To conclude, we observe that the mutation is then realized by varying the parameters orthogonally to the line $\Delta^{ij}_k=0$ which geometrically can be understood as varying the separation of the hyperplanes along $H_i$ and $H_j$ along the $x_k$ axis.
\end{proof}

\par The arguments in Theorem 5.1 can be viewed as capturing simple geometric properties of the hyperplane arrangement of the network. To illustrate this we explicitly realize a bifurcation from Figure~\ref{fig:bifs}.

\begin{example}

Consider the threshold-linear network given by the parameters:
\[
W=\small\begin{bmatrix}
0&   -0.63 &  -0.84\\
   -0.65  & 0 &  -0.67\\
   -0.45  & -0.50 & 0
\end{bmatrix}\hspace{.25cm} b=\begin{bmatrix}
0.43\\
    0.48\\
    0.41
\end{bmatrix}
\]
The directed graph $G$ associated to this network is of type $D_2$ from Figure~\ref{fig:bifs}. By computing the chirotope from the network parameters $W$ and $b$ and applying Theorem \ref{theorem1} we find the fixed point supports for this network are $\operatorname{FP}(W,b)=\{123\}$. To illustrate Theorem 5.1, consider the persistent bifurcation $\{123\}\to \{12\}$. Since $\{1,2\}\not\in\operatorname{FP}(W,b)$, we must have:
\[
\text{sgn}\:s^{12}_1=\text{sgn}\:s^{12}_2 = \text{sgn}\:s^{12}_3
\]
Thus, $s^{123}_3<0$ and the bifurcation $\{123\}\to \{12\}$ is obtained by varying the network parameters so that  $s^{123}_3>0$. This bifurcation arises because $1\leftrightarrow 2$ and both $1$ and $2$ are non-separating nodes in $G$ and thus $s^{123}_3$ is a realizable mutation. From the proof of the theorem with $(i,j,k)=(3,1,2)$, this mutation can be realized by varying the separation $\Delta_2^{31}$ as long as the condition in \eqref{caveat} holds:
\[
\text{sgn }\:s^{12}_2\cdot \Delta^{32}_1=-
\]
 Since $1\to 2$ in $G$, this is equivalent to $\Delta^{32}_1=b_2W_{31}-b_3W_{21}<0$. However, in our case we have
\[
\Delta^{32}_1=b_2W_{31}-b_3W_{21}=0.0505>0
\]
Thus, in order to unlock the mutation corresponding to the determinant $s^{123}_3$ we must first vary the separation $\Delta^{32}_1$ across zero. This can be done by increasing $W_{31}$ and decreasing $W_{21}$ and corresponds to pushing the hyperplane $H_2$ across $H_3$ along the $x_1$-axis (Figure~\ref{fig:clique_geometry_1}). 
\begin{figure}[h]
    \centering
    \includegraphics[width=6in]{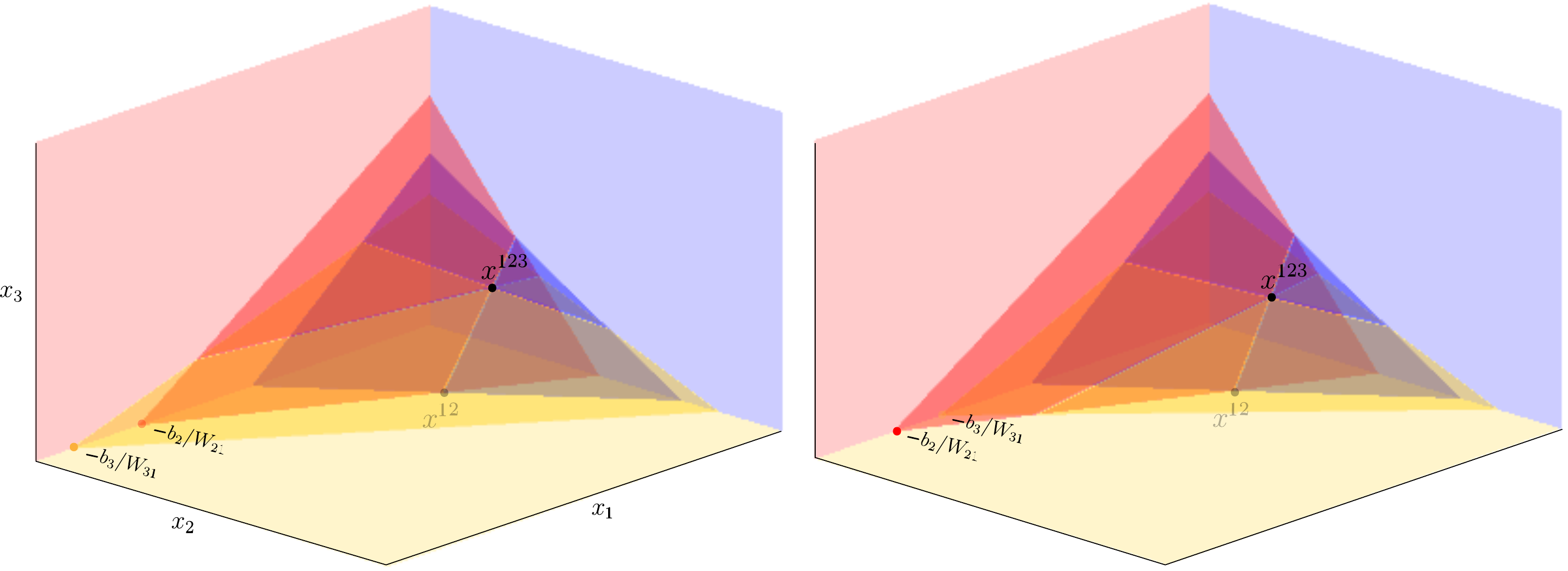}
    \caption{The sign of $\Delta^{23}_1$ is flipped by pushing the hyperplane $H_2$ over $H_3$ along the $x_1$-axis. This change must occur before the bifurcation $\{123\}\to\{12\}$ is possible.}
    \label{fig:clique_geometry_1}
\end{figure}
Once this change has been made, the theorem states that the bifurcation $\{123\}\to\{12\}$ is obtainable by varying the separation $\Delta^{13}_2$ via the parameters $W_{12}$ and $W_{32}$. Indeed, by increasing $W_{12}$ and decreasing $W_{32}$ we see first a simplicial cell appear in the arrangement corresponding to the mutation $s^{123}_3$ and then the collapse of this cell leading to the desired bifurcation $\{123\}\to\{12\}$.
\begin{figure}[h]
    \centering
    \includegraphics[width=6in]{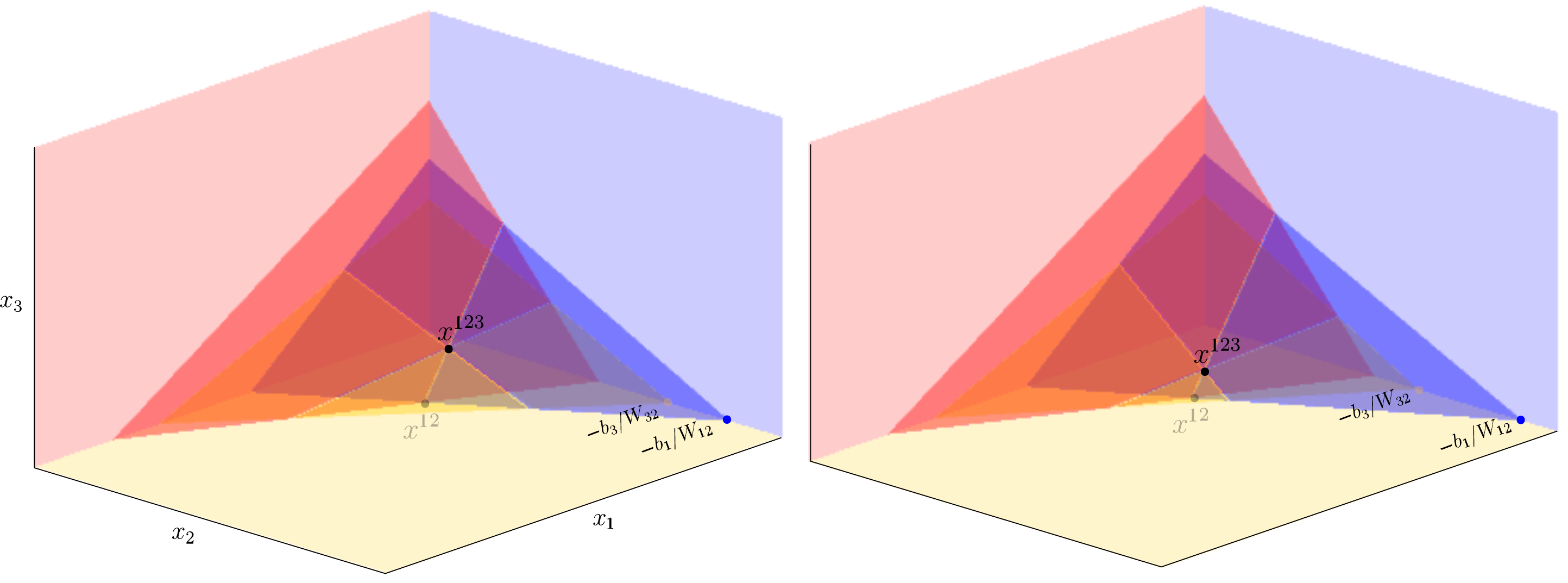}
    \caption{By increasing the separation of the hyperplanes $H_1$ and $H_3$ along the $x_2$-axis a simplicial cell is introduced into the arrangement corresponding to the bifurcation $\{123\}\to\{12\}$. By continuing to increase this separation the cell is collapsed and the bifurcation is realized.}
    \label{fig:my_label}
\end{figure}

\end{example}

\begin{example}
Consider the four dimensional threshold-linear network given by the parameters:
\[
W=\small\begin{bmatrix}
0 &  -0.89  & -0.83  & -0.56\\
   -0.89 &  0 &  -1.44 &  -1.38\\
   -1.59  & -0.74  & 0 &  -1.94\\
   -0.26 & -0.62 &  -0.04 &  0
  \end{bmatrix}
  \hspace{.5cm}
  b=\begin{bmatrix}
  0.46\\
    0.73\\
    0.85\\
    0.48
  \end{bmatrix}
   \]
The fixed point supports are $\operatorname{FP}(W,b)=\{14, 124, 1234\}$ and are found by computing the chirotope from the network parameters and then applying Theorem \ref{theorem1}.  By computing the directed graph we can gain an understanding of some of the support bifurcations that can be obtained by perturbing the network parameters away from $(W,b)$. For instance, we see that the graph contains the following subgraph on three nodes:
\begin{figure}[h]
    \centering
    \includegraphics[width=3in]{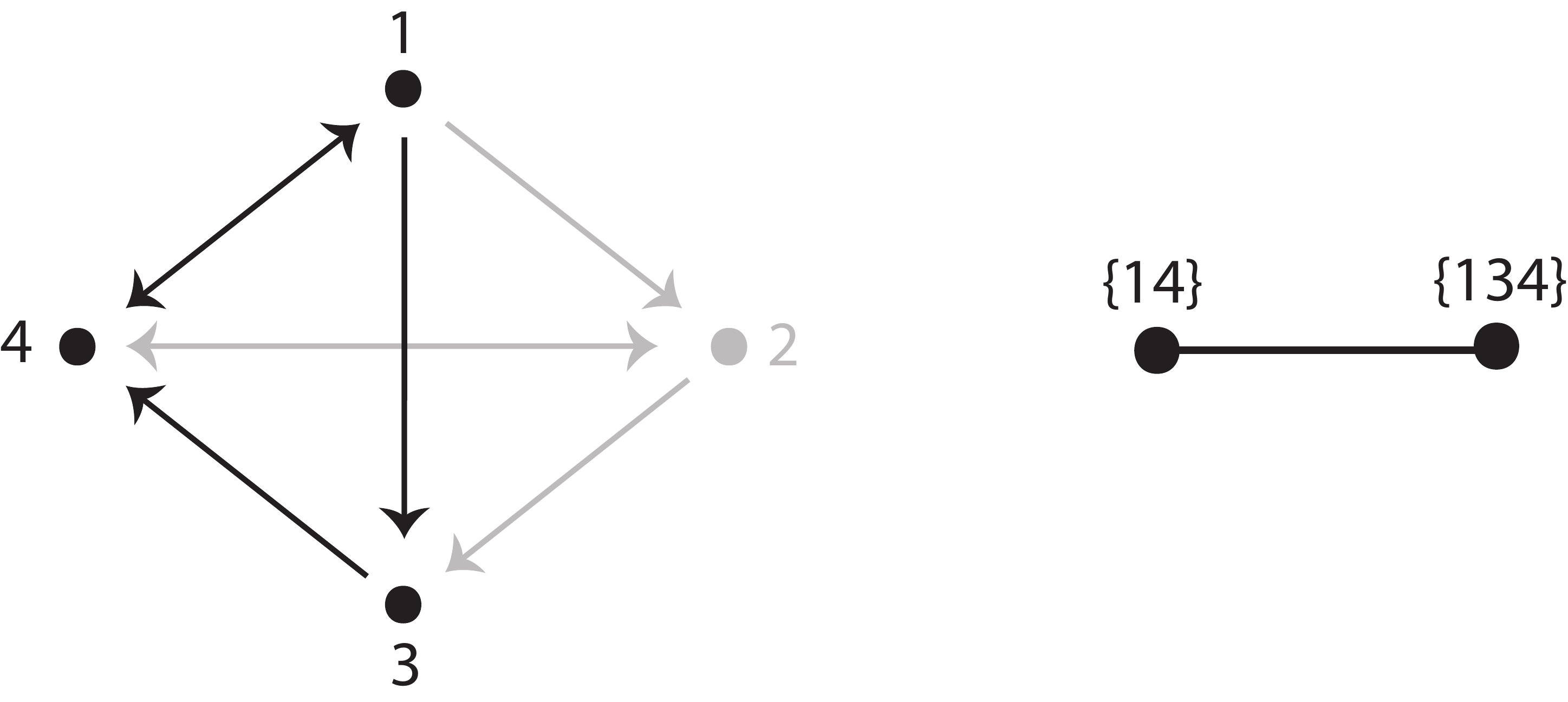}
    \caption{A subgraph of size three and its bifurcation graph.}
    \label{fig:n4_subgraph}
\end{figure}
\newline
This subgraph is isomorphic to graph $B4$ in Figure~\ref{fig:bifs} and after relabeling we see that its bifurcation graph consists of the single persistent bifurcation $\{14\}\to \{134\}$. From the results of this section, this bifurcation can be realized by varying the separation $\Delta^{34}_1$ corresponding to the non-separating neuron one. Consider continuously increasing the weight $W_{31}$. In order to preserve the directed graph we must maintain the inequality $W_{31}>-b_3/b_1=-1.85$ and so this is guaranteed if we are increasing $W_{31}$. What we observe in this case is not the bifurcation $\{14,124,1234\}\to\{134,124,1234\}$ we are interested in but the bifurcation $\{14,124,1234\}\to\{14\}$. Numerically, this bifurcation occurs near $W_{31}\sim-1.26$. Continuing, we observe the bifurcation $\{14\}\to \{134\}$ near $W_{31}\sim-.14$ To understand this, look at the bifurcation graph of the full four dimensional network.
\begin{figure}[h]
    \centering
    \includegraphics[width=5in]{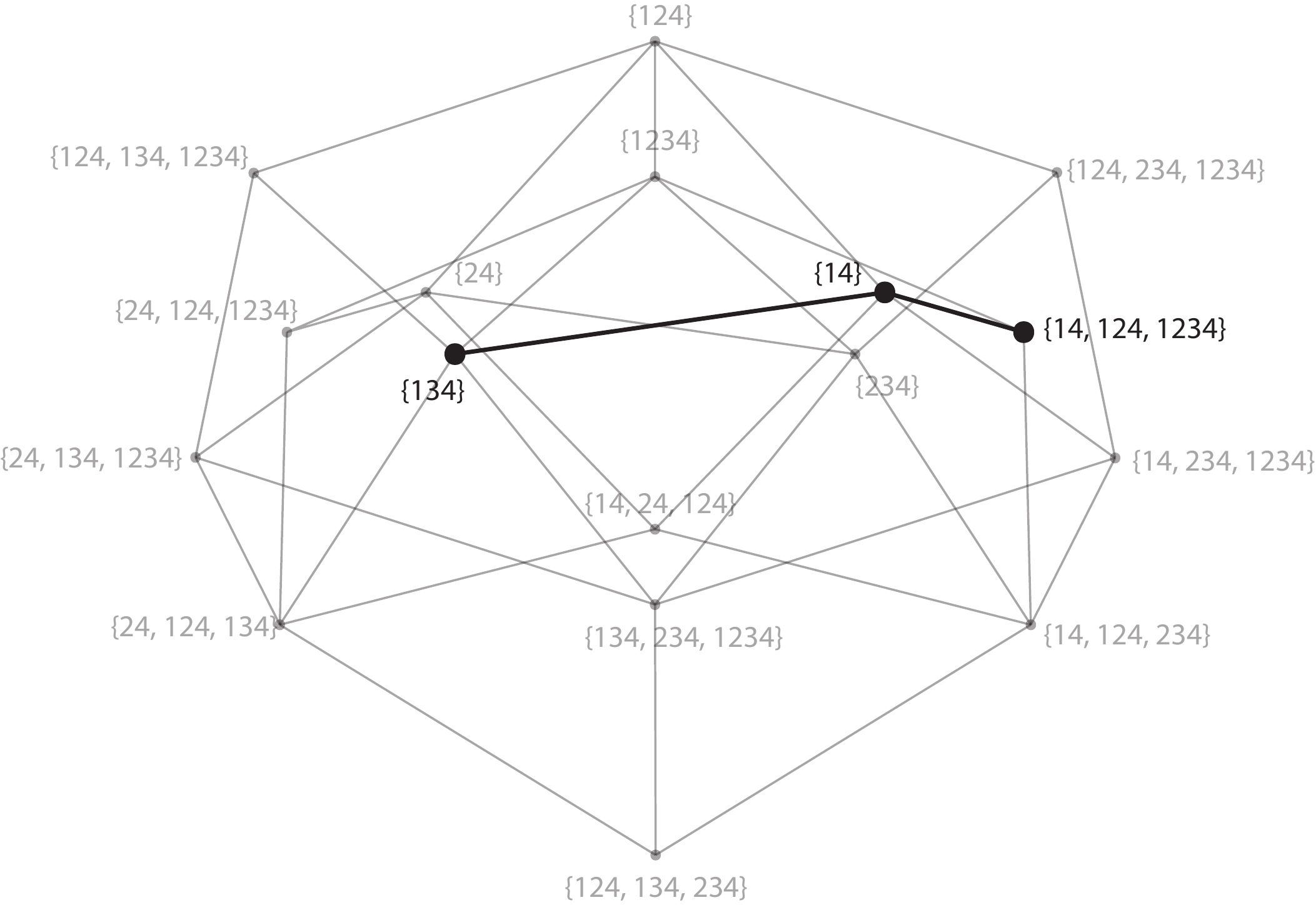}
    \caption{Bifurcation graph of the directed graph of size four in Figure~\ref{fig:n4_subgraph}. The highlighted edges correspond to the bifurcation $\{14\}\to\{134\}$ arising from the subgraph induced by the neurons $\{1,3,4\}$.}
    \label{fig:n4_bif}
\end{figure}
The node $\{14,124,1234\}$ in the bifurcation graph represents an equivalence class of networks containing the particular network given by the parameters $W$ and $b$ above. Note that the single bifurcation $\{14\}\to\{134\}$ of the subgraph does not correspond to an outgoing edge of this node. Thus, we see that as we try to realize the bifurcation $\{14\}\to\{134\}$ we must first go through the bifurcation $\{124,1234\}\to \emptyset$. This bifurcation cannot be viewed as a bifurcation of the subgraph and an understanding of why this bifurcation appears in the graph would require the analogue of Theorem \ref{theorem2} in dimension four.
\par The relationship between the bifurcation graph of a network and that of the smaller motifs from which it is constructed is non-trivial. In general, all that can be said is that if a bifurcation involving the neurons of a subset $\sigma$ exists, then it must exist in the subgraph induced by the subset $\sigma$. Moreover, we see from this example that the bifurcation structure of even small networks of size four can already be quite complex and that this type of bifurcation analysis in higher dimensions likely will require additional constraints besides what it is imposed by the directed graph.
\end{example}